\documentclass[reqno]{amsart}

\usepackage{textcmds} 
\usepackage{amsmath, amssymb, amsfonts, amstext, verbatim, amsthm, mathrsfs}
\usepackage[mathcal]{eucal}
\usepackage{microtype}
\usepackage[all]{xy}
\usepackage[modulo]{lineno}
\usepackage[usenames]{color}
\usepackage{aliascnt}
\usepackage{enumitem}
\usepackage{xspace}
\usepackage{amsfonts}
\usepackage{amssymb}
\usepackage[centertags]{amsmath}
\usepackage{amsthm}
\usepackage{dsfont}
\usepackage{bm}
\usepackage{color}
\usepackage{subfigure}
\usepackage{amsmath}
\usepackage{array}
\usepackage[all]{xy}

\usepackage{graphics,graphicx} 

\usepackage{xfrac}

\usepackage{setspace}

\let\oldtocsection=\tocsection

\let\oldtocsubsection=\tocsubsection

\renewcommand{\tocsection}[2]{\hspace{0em}\oldtocsection{#1}{#2}}
\renewcommand{\tocsubsection}[2]{\hspace{1em}\oldtocsubsection{#1}{#2}}

\usepackage[backref,colorlinks=true,linkcolor=blue,citecolor=red,urlcolor=blue,citebordercolor={0 0 1},urlbordercolor={0 0 1},linkbordercolor={0 0 1}]{hyperref} 

\newcommand{\Sk}{L_\mathrm{Sk}}

\newcommand{\C}{\mathbb{C}}

\newcommand{\R}{\mathbb{R}}
\newcommand{\Q}{\mathbb{Q}}
\newcommand{\Z}{\mathbb{Z}}
\newcommand{\N}{\mathbb{N}}

\newcommand{\CP}{\mathbb{C}P}

\newcommand{\BlI}{\mathbb{C}P^2\# \overline{\mathbb{C}P^2}}
\newcommand{\BlII}{\mathbb{C}P^2\# 2\overline{\mathbb{C}P^2}}
\newcommand{\BlIII}{\mathbb{C}P^2\# 3\overline{\mathbb{C}P^2}}

\newcommand{\vol}{\mathrm{vol}}

\newcommand{\del}{\partial}

\newcommand{\imp}{\Rightarrow}

\newcommand{\be}{\begin{equation}}
\newcommand{\ee}{\end{equation}}
\newcommand{\bea}{\begin{eqnarray}}
\newcommand{\eea}{\end{eqnarray}}
\newcommand{\beastar}{\begin{eqnarray*}}
\newcommand{\eeastar}{\end{eqnarray*}}

\newtheorem{thm}{Theorem}[section]

\newtheorem{lem}[thm]{Lemma}
\newtheorem{prop}[thm]{Proposition}

\newtheorem{conj}[thm]{Conjecture}
\newtheorem{prb}[thm]{Problem}

\theoremstyle{definition}
\newtheorem{defn}[thm]{Definition}
\newtheorem{rmk}[thm]{Remark}

\newtheorem{ques}[thm]{\rm\bfseries{Question}}

\newtheorem*{thm*}{Theorem}

\numberwithin{equation}{section}

\def\R{{\mathbb R}}

\def\E{{\mathbb E}}
\def\Z{{\mathbb Z}}
\def\C{{\mathbb C}}
\def\R{{\mathbb R}}
\def\P{{\mathbb P}}

\def\N{{\mathbb N}}

\def\11{{\mathbb I}}

\def\C{\mathbb{C}}
\def\Z{\mathbb{Z}}

\def\Q{\mathbb{Q}}

\def\E{\ifmmode{\mathbb E}\else{$\mathbb E$}\fi} 
\def\N{\ifmmode{\mathbb N}\else{$\mathbb N$}\fi} 
\def\R{\ifmmode{\mathbb R}\else{$\mathbb R$}\fi} 
\def\Q{\ifmmode{\mathbb Q}\else{$\mathbb Q$}\fi} 
\def\C{\ifmmode{\mathbb C}\else{$\mathbb C$}\fi} 
\def\Z{\ifmmode{\mathbb Z}\else{$\mathbb Z$}\fi} 
\def\P{\ifmmode{\mathbb P}\else{$\mathbb P$}\fi} 
\def\CS{\ifmmode{\mathbb S}\else{$\mathbb S$}\fi} 
\def\DD{\ifmmode{\mathbb D}\else{$\mathbb D$}\fi} 

\def\R{{\mathbb R}}

\def\E{{\mathbb E}}
\def\Z{{\mathbb Z}}
\def\C{{\mathbb C}}
\def\R{{\mathbb R}}

\def\N{{\mathbb N}}

\def\CS{{\mathcal S}}

\def\CU{{\mathcal U}}
\def\CV{{\mathcal V}}

\def\darr#1{\raise1.5ex\hbox{$\leftrightarrow$}
\mkern-16.5mu #1}

\def\roughly#1{\raise.3ex\hbox{$#1$\kern-.75em
\lower1ex\hbox{$\sim$}}}

\def\opname#1{\mathop{\kern0pt{\rm #1}}\nolimits}

\def\dim{\opname{dim}}
\def\vol{\opname{vol}}

\def\Int{\operatorname{Int}}

\begingroup
\makeatletter
\@for\theoremstyle:=definition,remark,plain,TheoremNum\do{%
\expandafter\g@addto@macro\csname th@\theoremstyle\endcsname{%
\addtolength\thm@preskip\parskip
}%
}
\endgroup

\begin{document}

\quad \vskip1.375truein

\thispagestyle{empty}
\title[Exotic Lagrangian tori $T_{a,b,c}$]
{Asymptotic behavior of exotic Lagrangian tori $T_{a,b,c}$ in $\CP ^2$ as $a+b+c \to \infty$}

\author{Weonmo Lee, Yong-Geun Oh, Renato Vianna}

\thanks{WL and YO were supported by the IBS project IBS-R003-D1.
RV was supported by the Brazil's National Council of scientific and technological
development CNPq, via the `bolsa de produtividade' fellowship, and by the Serrapilheira fellowship}

\address{Weonmo Lee\\
Department of Mathematics, POSTECH, Pohang, Korea \&
Center for Geometry and Physics, Institute for Basic Sciences (IBS), Pohang, Korea} \email{wm.lee@postech.ac.kr}
\address{Yong-Geun Oh\\
Center for Geometry and Physics, Institute for Basic Sciences (IBS), Pohang, Korea \& Department of Mathematics,
POSTECH, Pohang, Korea} \email{yongoh1@postech.ac.kr}
\address{Renato Vianna\\
Institute of Mathematics, Federal University of Rio de Janeiro (UFRJ), Rio de Janeiro, Brazil}
\email{renato@im.ufrj.br}

\date{}

\begin{abstract} In this paper, we study various asymptotic behavior of the infinite family of
monotone Lagrangian tori $T_{a,b,c}$ in $\CP^2$ associated to Markov triples $(a,b,c)$ described
in \cite{Vi14}.
We first prove that the Gromov capacity of the complement
$\CP^2 \setminus T_{a,b,c}$ is greater than or equal to $\frac13$ of the area of the
complex line for all Markov triple $(a,b,c)$. We then prove that there is a representative of
the family $\{T_{a,b,c}\}$ whose loci completely miss a metric ball of nonzero size and
in particular the loci of the union of the family is not dense in $\CP^2$.
\end{abstract}



\maketitle

\tableofcontents

\begin{center}
\section{Introduction}
\end{center}

In \cite{Vi13,Vi14}, the third named author constructed an interesting family of infinitely many
monotone Lagrangian tori in $\CP^2$ by constructing a monotone Lagrangian torus
associated to each of the Markov triples $(a,b,c)$,
i.e., positive integers satisfying the equation
\be
a^2+b^2+c^2=3abc.
\ee
For his construction, it was used constructions almost toric fibration, Symington's nodal surgery operations
\cite{Sy03} or the operation of
\emph{rational blow down} of the weighted projective planes $\CP(a^2,b^2,c^2)$
to $\CP^2$. Denote by $T_{a,b,c}$ any realization of the torus associated to the triple $(a,b,c)$
in its Hamiltonian isotopy class in $\CP^2$.
To show that this family of $T_{a,b,c}$ are not pairwise Hamiltonian isotopic to one another,
he used the disc-counting invariants which are known to be well-defined for \emph{monotone} Lagrangian tori
\cite{ElPo93,Oh95}.

The starting point of our research in the present article lies in our attempt to
understand the tori $T_{a,b,c}$ in terms of the geometry of Fubini-Study metric on $\CP^2$.
(We refer to Section \ref{sec:discussion} for more discussion on the related
questions.) As a first step towards this goal, we ask the following question

\medskip

\begin{ques}\label{ques:ergodic} Let $\{T_{a,b,c}\}$ be a fixed family of tori in $\CP^2$.
What is the geometric behavior of these tori as $a, \, b, \, c \to \infty$?
For example, will the tori densely spread out $\CP^2$ as $a+b+c \to \infty$?
\end{ques}

We denote by $\frak M$ the set of Markov triples. We remark that a specific construction of
the $\{T_{a,b,c}\}$ tori depends on various unspecified parameters.
Because of the way how they are constructed, it is not easy to visualize the tori
in the Fubini-Study metric of $\CP^2$ although they are well-defined up to Hamiltonian isotopy
on $(\CP^2,\omega_{\text{\rm FS}})$.

Fix any smooth metric on $\CP^2$, e.g., take the Fubini-Study metric of $\CP^2$.

\begin{ques}\label{ques:main} Let $\{T_{a,b,c}\}$ be any realization of the family of monotone Lagrangian tori in
$(\CP^2,\omega_{\text{\rm FS}})$. Consider the following asymptotic quantity
\be\label{eq:delta}
\delta:=\inf_{(a,b,c) \in \frak M} \sup_{x \in \CP^2\setminus T_{a,b,c}}
d(x,T_{a,b,c})
\ee
where $d(x,T_{a,b,c})$ is the distance from $x$ to $T_{a,b,c}$. Is $\delta > 0$? If so,
estimate this $\delta$.
\end{ques}

More intuitively and equivalently, the question asks if for any given point $x \in \CP^2$
and a positive constant $\delta > 0$, there exists a Markov triple $(a,b,c)$ such that
$B_\delta(x) \cap T_{a,b,c} \neq \emptyset$ for the given family $\{T_{a,b,c}\}$.

This number $\delta$ is not a priori a symplectic invariant. More precisely if $T'_{a,b,c}$ is another
realization of these tori, this quantity may vary. Because of this, we
consider the following quantity
\be\label{eq:asymptotics}
\inf_{(a,b,c) \in \frak M} c_G(\CP^2; T_{a,b,c})
\ee
where $c_G(\CP^2; T_{a,b,c})$ is the \emph{relative Gromov area}
\be\label{eq:e(abc)}
c_G(\CP^2; T_{a,b,c}): = \sup_e \{ \pi r^2 \mid e: B^4(r) \to \CP^2 \setminus T_{a,b,c} \text{
is a symplectic embedding}\}
\ee
Relative Gromov area is a symplectic invariant and has been systematically studied by Biran
and Biran-Cornea \cite{Bir01,Bir06,BC09A} for general pair of symplectic manifold $(M,\omega)$ and its Lagrangian submanifold $L$.

  \begin{rmk} We warn the readers that this definition of \emph{relative Gromov area}
  is not the one used by Biran and Cornea in \cite{BC09A}. For our purpose in the present paper, we do not need their finer version and so we will just use the same term instead of introducing
  another different term for the Gromov area of the complement.
  \end{rmk}

The first theorem we prove in the present paper is the following rather optimal lower bound.
(See the construction given in Section \ref{sec:balls}.)

\medskip

\begin{thm}\label{thm:lowerbound} Let $T_{a,b,c}$ be any of monotone Lagrangian tori of
  \cite{Vi14}. Then
$$
\inf_{(a,b,c) \in \mathfrak M} c_G(\CP^2; T_{a,b,c}) \geq \frac{2\pi}{3}.
$$
Here we normalize the Fubini-Study form $\omega_{\text{\rm FS}}$ so that
the area of the complex line is $2\pi$.
\end{thm}

Motivated by the nature of our construction given in Section \ref{sec:balls},
we conjecture that the equality holds in the above theorem.

On the other hand, this lower bound is certainly not optimal for an individual torus.
For example for the case of Clifford torus corresponding to $(a,b,c) = (1,1,1)$,
it is easy to see $c_G(\CP^2; T_{1,1,1}) \geq \frac{4\pi}{3}$ and Biran-Cornea \cite{BC09A} proved
$c_G(\CP^2; T_{1,1,1}) \leq \frac{4\pi}{3}$, and hence
\be\label{eq:T111}
c_G(\CP^2; T_{1,1,1}) = \frac{4\pi}{3}.
\ee
This leads us to a very interesting open problem
\begin{prb}\label{prb:Tabc} Find the precise
estimate of $c_G(\CP^2; T_{a,b,c})$ as done for the Clifford torus.
\end{prb}

The above theorem still does not prevent the loci of the union of the family $\{T_{a,b,c}\}$
being dense in $\CP^2$, it may happen that
there exists a family of symplectic balls of nonzero size associated to Markov triples
$(a,b,c)$ which are \emph{stretched thin and wildly spread} around the ambient space $\CP^2$
without touching the corresponding $T_{a,b,c}$ torus respectively.

\begin{thm} \label{thm:not-dense} There exists a family $\{T_{a,b,c}\}$ of tori
that misses some closed metric ball of non-zero size in $\CP^2$. In fact, the supremum of
the Gromov areas of such metric balls is $\frac{2\pi}3$. In particular, the loci of
the family is not dense in $\CP^2$.
\end{thm}

The proof of this theorem will be given in Section \ref{sec:dense} using the
geometric mutation theorem of the \emph{Lagrangian seeds} studied in
in \cite{STW16} and \cite{To17,PaTo17}. But one only needs to look at Figure
\ref{fig:All_ATFs} to see how it goes.

In the rest of the paper, we will provide various estimates relevant to
the ball packing problem in $\CP^2 \setminus T_{a,b,c}$ or $(\CP^2 \setminus E) \setminus T_{a,b,c}$
where $E \subset \CP^2$ is a smooth cubic curve, which corresponds to a
Donaldson divisor of $\CP^2$. One of the outcomes is the following result

\begin{thm}\label{thm:Chekanov-Schlenk} Any $T_{a,b,c}$ tori, in particular the Chekanov torus,
can be embedded into the monotone $\CP^2 \# \overline{k\CP^2}$ for $k \leq 5$.
\end{thm}
This in particular affirmatively answers to a question posed by Chekanov and Schlenk \cite[Section 7]{ChSch10}
which asks whether Chekanov torus can be embedded into $\CP^2 \# \overline{3\CP^2}$.

In Section \ref{sec:discussion}, we make further
discussion and propose several open questions related to the geometry of
the tori $T_{a,b,c}$.

\section{Review of the exotic tori $T_{a,b,c}$}

  The third named author \cite{Vi13, Vi14} constructed a family of
  infinitely many non-Hamiltonian isotopic monotone Lagrangian tori in $\CP^2$
  as the transfers to  $\CP^2$ of the fibers $T(a^2,b^2,c^2)$ at the (labeled) barycenter of the moment polytope of the weighted projective plane $\CP(a^2,b^2,c^2)$ or its relevant almost toric fibration. He utilized
  Symington's symplectic rational blow-down operations \cite{Sy01}
  on each neighborhood of orbifold points thereof and Moser's deformation of the glued symplectic forms on the resulting
  blow-down to the Fubini-Study form on $\CP^2$ for his construction. For the simplicity of notation, we denote
  by
  $$
  T_{a,b,c} \subset \CP^2
  $$
  any realization of the family of these tori in their Hamiltonian isotopy class in $\CP^2$. We exclusively
  reserve $T(a^2,b^2,c^2)$ for the fiber at the barycenter of the moment polytope of $\CP(a^2,b^2,c^2)$.
  The torus $T_{a,b,c}$ can be also realized as the fiber of a base point of an almost toric fibration of $\CP^2$ \cite{Vi14}.
  An almost toric fibration is a singular Lagrangian fibration with nodal singular fibers.
  Here a nodal singular fiber carries an isolated singularity whose image under the almost toric projection
  lies on the interior of the base diagram of the almost toric fibration. This image point is called a \emph{node}.

\subsection{Almost toric fibration and nodal surgeries}

  In this subsection, we recall definitions of almost toric fibration, nodal surgery operation and related results from \cite[Section 2.3]{Vi13}.

  \begin{defn}[\cite{Zu97}, \cite{Vi13}]
  An \emph{almost toric fibration} of a symplectic four manifold $(M,\omega)$ is
  a Lagrangian fibration $\pi: (M, \omega) \rightarrow B$ such that any point of
  $(M, \omega)$ has a Darboux neighborhood (with symplectic form $dx_1\wedge dy_1 + dx_2\wedge dy_2$) in which the map $\pi$ has one of the following forms:
  \begin{align*}
  \pi(x,y) & =  (x_1, x_2), &&\text{regular point}, \\
  \pi(x,y) & =  (x_1, x_2^2 + y_2^2), &&\text{elliptic, co-rank one}, \\
  \pi(x,y) & =  (x_1^2 + x_2^2, x_2^2 + y_2^2), &&\text{elliptic, co-rank two}, \\
  \pi(x,y) & =  (x_1y_1 + x_2y_2, x_1y_2 - x_2y_1), &&\text{nodal or focus-focus},
  \end{align*}
  with respect to some choice of coordinates near the image point in $B$.
  An \emph{almost toric manifold} is a symplectic manifold equipped with an almost toric fibration.
  \end{defn}

A Lagrangian fibration induces an integral affine structure $\Lambda$ on the base $B$ with singularity.
  Such pair $(B,\Lambda)$ is called an \emph{almost toric base} \cite{Sy03}.
  For an almost toric manifold, there is a nontrivial monodromy around the nodal singular fiber.
  This prevents one from embedding the full almost toric base into $(\R^2,\Lambda _0)$
  where $\Lambda_0$ is a standard integral affine structure.
  However removing an embedded curve $R$ joining a point of the boundary, in particular the vertex, of a moment polytope and the node,
  called a \emph{branch curve}, makes this embedding possible.

  \begin{defn}
  Suppose we have an integral affine embedding $\Phi: (B - R, \Lambda) \to (\mathbb{R}^2, \Lambda_0)$,
  where $(B, \Lambda)$ is an almost toric base and $R$ is a set of branch curves. A \emph{base
  diagram} of $(B, \Lambda)$ with respect to $R$ and $\Phi$ is the image of $\Phi$ decorated with
  the following data:

  \begin{itemize}
  \item an \emph{x} marking the location of each node and
  \item  dashed lines indicating the portion of $\del\overline{\Phi(B - R)}$ that
  corresponds to $R$.
  \end{itemize}
  \end{defn}

  If the direction of $R$ is $(k,l)$, then the monodromy around the node can be represented by
  $$
  A_{(k,l)} = \begin{pmatrix} 1 -kl & k^2 \\-l^2 & 1 + kl \end{pmatrix}
  $$
  with respect to some choice of basis. (See \cite{Vi13} for the details.)

  Consider a moment polytope in $(\R^2)^*.$
  Nodal trade introduces a node and a cut inside the moment polytope in place of the vertex.(See the first two triangles in Figure \ref{fig:1})
  The corresponding fibration has a nodal singular fiber and becomes an almost toric fibration.
  Nodal slide literally ``slides'' a node along an eigenline of the monodromy map.(See the second and the third triangles in Figure \ref{fig:1})

  \begin{figure}[h]
  \centering
  \includegraphics[scale=1.0]{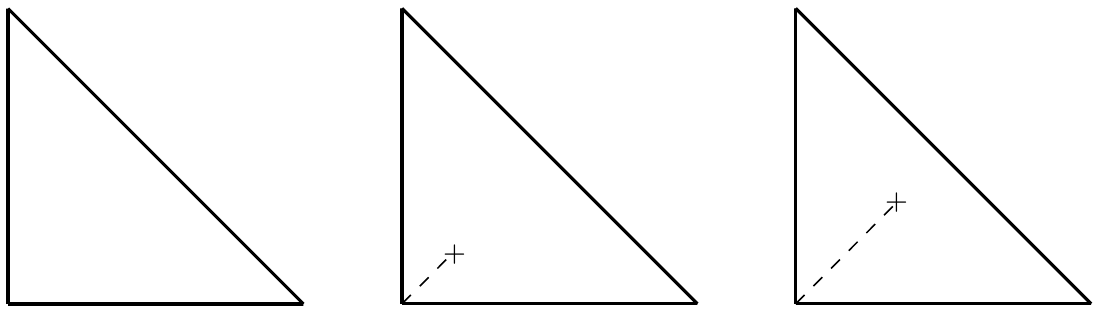}
  \caption{Almost toric surgeries on the moment polytope of $\CP^2$}
  \label{fig:1}
  \end{figure}

Following Symington \cite{Sy03}, we consider the following operations on almost toric
fibrations which do not change the symplectic structure of the total space up to symplectomorphism.
  \begin{defn}
  Let $(B,\Lambda_i)$ be two almost toric bases, $i = 1, 2$. We say that
  $(B,\Lambda_1)$ and $(B,\Lambda_2)$ are related by a \emph{nodal slide}
  if there is a curve $\gamma$ in $B$ such that
  \begin{itemize}
    \item $(B - \gamma,\Lambda_1)$ and $(B - \gamma,\Lambda_2)$
  are isomorphic,
    \item $\gamma$ contains one node of $(B,\Lambda_i)$ for each $i$ and
    \item $\gamma$ is contained in the \emph{eigenline} (line preserved by the monodromy)
     through that node.
     \end{itemize}
  \end{defn}

  \begin{defn}
  Let $(B_i,\Lambda_i)$ be two almost toric bases, $i = 1, 2$. We say that
  $(B_1,\Lambda_1)$ and $(B_2,\Lambda_2)$ differ by a \emph{nodal trade}
  if each contains a curve $\gamma_i$ starting at $\del B_i$ such that
  $(B_1 - \gamma_1,\Lambda_1)$ and $(B_2 - \gamma_2,\Lambda_2)$
  are isomorphic, and $(B_1, \Lambda_1)$ has one less vertex than $(B_2,\Lambda_2)$.
  \end{defn}

  It is shown by Symington \cite{Sy03} that
  these almost toric operations do not change the diffeomorphism type they represent and
  keep the symplectic structure up to isotopy and so Moser's argument
  shows that the two almost toric manifolds before and after the operations are symplectomorphic.

\subsection{$\CP(a^2,b^2,c^2)$ as a symplectic reduction of $\C^3$}

The base diagram for an almost toric fibration of $\CP^2$ having a monotone Lagrangian torus $T_{a,b,c}$ at its barycenter looks like the moment polytope of a weighted projective plane $\CP(a^2,b^2,c^2)$, except that the diagram
is equipped with nodes and cuts.

In this subsection we describe various aspects of geometry of $\CP(a^2,b^2,c^2)$ as a toric
orbifold.
We consider the following $\C^*$ action of $S^1$ on $\C^3 \setminus \{0\}$ defined by
  \be\label{eq:zeta}
   \zeta \cdot (x,y,z) := (\zeta^{a^2}x,\zeta^{b^2}y, \zeta^{c^2}z)
  \ee
  for $\zeta \in \C \setminus 0$.
  The \emph{weighted projective plane} $\CP(a^2,b^2,c^2)$ \emph{with weights} $(a^2,b^2,c^2)$
  as a complex orbifold
  is nothing but the quotient of $(\C^*)^3$ by this action.
  We denote by $[x:y:z]$ its element represented by $(x,y,z) \in \C^3 \setminus \{0\}$.
$\CP(a^2,b^2,c^2)$ is an orbifold with three orbifold points $[1:0:0],[0:1:0],[0:0:1].$
  Their corresponding orbifold structure groups are
  $$
  \Z / a^2\Z, \quad \Z / b^2\Z,\quad \Z / c^2\Z,
  $$
  respectively.

For our purpose of proving Main Theorem, we need to explicitly express
the relevant symplectic structure and $T^2$-action on $\CP(a^2,b^2,c^2)$
starting from the linear sigma model construction as in \cite{Wit93, Aud_book}, which
was exploited in the Lagrangian Floer theory of toric manifolds in \cite[Section 3]{CO06}.
We regard the toric orbifold $\CP(a^2,b^2,c^2)$ as the symplectic reduction $\phi_K^{-1}(r_0)/K$
of $\C^3$ under the action of the circle subgroup $K \subset T^3$ for a suitable choice of
$r_0 \in \R$ where $\phi_K: \C^3 \to \R$ is the associated moment map.
Then it carries the canonical action of the residual torus $T^3 /K \cong T^2$ thereon.
We will call this particular $T^2$-action the \emph{residual $T^2$-action} on
$\phi_K^{-1}(r_0)/K \cong \CP(a^2,b^2,c^2)$.

Now we need to describe this torus action on $\CP(a^2,b^2,c^2)$ and its moment polytope
explicitly, employing notations from \cite[Section 2.2]{Abr01}. For this purpose, we start with the standard torus action of $T^3$ on
$\left(\C^3,\sum_{i=1}^{d=3} du_{i} \wedge dv_{i}\right)$ defined by
  $$
  \theta \cdot (z_{1},z_{2},z_{3})=(e^{i\theta_1}z_{1},e^{i\theta_2}z_{2},e^{i\theta_3}z_{3}).
  $$
This $T^3$ action
  has its moment map $\phi_{T^3}: \C^3 \to \R^3$ given by
  $$
  \phi_{T^3}(z_{1},z_{2},z_{3})=\sum_{i=1}^{3} \frac{|z_i|^2}{2} e_{i}^{*} + \lambda
  $$
  for an arbitrary choice of constant vector $\lambda \in (\R^{3})^{*}$ in general
  where  $\{e_{1}^{*},e_{2}^{*},e_{3}^{*}\}$ is the basis of $(\R^{3})^{*}$ dual to the standard basis $\{e_1,e_2,e_3\}$ of
  $\R^3$.
  Setting $\lambda = \sum_{i=1}^{3} \lambda_{i} e_{i}^{*}$ with
  $$
  \lambda_1=-b^{2}c^{2}, \, \lambda_2=\lambda_3=0
  $$
  we have
  $$
  \phi_{T^3}(z_1,z_2,z_3)=\left(\frac{|z_{1}|^2}{2}-b^{2}c^{2}\right)e_{1}^{*} + \frac{|z_{2}|^2}{2}e_{2}^{*} + \frac{|z_{3}|^2}{2}e_{3}^{*}.
  $$

 Being a subgroup, any circle subgroup $K \subset T^3$ naturally acts on $\C^{3}$ with moment map
  $$
  \phi_{K} = \iota^{*} \circ \phi_{T^3} = \sum_{i=1}^{3}\left(\frac{|z_i|^2}{2}+ \lambda_{i}\right) \iota^{*}(e_{i}^{*}) \in \mathfrak{k}^{*}.
  $$
   Then the symplectic quotient $M_P := \phi_{K}^{-1}(0)/K$ of $\C^3$
  carries the canonical reduced symplectic form and
  the residual torus action by $T^3/K \cong T^2$  whose moment map image
  is the labeled polytope $P$ described in subsection \ref{subsec:moment polytope}. Furthermore
  this reduced space is precisely the symplectic orbifold $\CP(a^2,b^2,c^2)$
  equipped with the 2-torus action by the torus $T^2 \cong T^3/K$.

  We now identify what this circle subgroup $K \subset T^3$ associated to $\CP(a^2,b^2,c^2)$ is.
  Define a linear map $\beta:\R^3 \to \R^2$ by
  $$
  \beta(e_{i})=m_{i}\mu_{i} \text{ for } i=1,2,3
  $$
  and denote by $\mathfrak{k}$ the kernel of $\beta$.
  We have the short exact sequences
  $$
  0 \to \mathfrak{k} \overset{\iota}{\to} \R^3 \overset{\beta}{\to} \R^2 \to 0
  \ \ \ \mbox{and its dual}\ \ \
  0 \to (\R^2)^{*} \overset{\beta^{*}}{\to} (\R^3)^{*} \overset{\iota^{*}}{\to}\mathfrak{k}^{*} \to 0\ .
  $$
  Denote by $K \subset T^3$ the subgroup generated by $\mathfrak{k}$.
  $\beta$ also induces the exact sequence of abelian groups
  $$
  0 \to K \to \R^3/ (2\pi\Z)^3 \to \R^2/(2\pi \Z)^2 \to 0.
  $$
  For any element $\theta=(\theta_1,\theta_2,\theta_3)$ in $\mathfrak k$, since $m_i =1$ for all $i=1,2,3$,
  $$
  \Sigma m_{i}\theta_{i} \mu_{i}=\Sigma \theta_{i} \mu_{i} \in (2\pi \Z)^3.
  $$
  A simple computation shows
  \begin{lem}\label{lem:k} The one-dimensional integral sub-lattice
   $\mathfrak k$ is generated by $(a^2,b^2,c^2)$.
  \end{lem}
  \begin{proof}
   Any element $\theta=(\theta_1,\theta_2,\theta_3)$ in $\mathfrak k$ satisfies $$\Sigma \theta_{i} \mu_{i}=(-(bl_{2}-1)\theta_{1}-(al_{1}-1)\theta_{2}+\theta_{3},-b^{2}\theta_{1}+a^{2}\theta_{1}) \equiv (0,0).$$
   From the second slot, $$\theta_2=\frac{b^2}{a^2}\theta_1.$$
   This implies
   \begin{eqnarray*}
   \theta_3 &=& (bl_2-1)\theta_1+(al_1-1)\frac{b^2}{a^2}\theta_1 \\
   &=& \frac{\theta_1}{a^2}(a^2bl_2-a^2+ab^2l_1-b^2) \\
   &=& \frac{c^2}{a^2}\theta_1.
   \end{eqnarray*}

   Here the last equality comes from equalities $$a^2bl_2-a^2+ab^2l_1-b^2=ab(al_2+bl_1)-a^2-b^2$$ and \eqref{eq:al2+bl1=3c} and the fact that $(a,b,c)$ is a Markov triple.
   Therefore $\mathfrak k$ is generated by $(a^2,b^2,c^2)$.
 \end{proof}
  Therefore this lemma is consistent with the $\C^*$ action \eqref{eq:zeta}.

\subsection{Residual $T^2$-action on $\CP(a^2,b^2,c^2)$ and its moment polytope} \label{subsec:moment polytope}

  Using the generalization of Delzant's argument applied to toric orbifold, Lerman and Tolman \cite{LT97}
  describe its associated orbifold moment map $\phi_P$ and the associated symplectic form $\omega_P$ on $\phi_{P}^{-1}(\text{\rm Int}\,P)$ explicitly, which we now recall.
  In order to apply their argument, we borrow relevant definitions and arguments from the
  exposition of Abreu \cite[Section 2.2]{Abr01} now.

  \begin{defn}
  A convex polytope $P$ in $(\R^n)^*$ is said to be \emph{simple} and \emph{rational} if
  \begin{itemize}
  \item $n$ facets meet at each vertex $p$,
  \item those edges meeting at the vertex $p$ are all rational, i.e., each edge has the form $p+tv_i$  where $0 \leq t \leq \infty, v_i \in (\Z^n)^*$and
  \item the $v_1,\cdots,v_n$ can be chosen to be a $\Q-$basis of the lattice $(\Z^n)^*.$
  \end{itemize}
  \end{defn}

  We call the polytope $P$ a \emph{labeled polytope} if it is a rational simple convex polytope and there is a positive integer label on the interior of each facet.

  \begin{thm} \emph{\cite{LT97, Abr01}} \label{LT}
  Let $(M,\omega)$ be a compact symplectic toric orbifold, with moment map $\phi : M \to (\R^n)^*.$
  Then $P=\phi(M)$ is a labeled polytope.
  For each facet $F$ of $P$, there exists a positive integer $m_F$, the \emph{label} of $F$, such that the orbifold structure group of every point in $\phi^{-1}(\text{\rm Int}\,F)$ is $\Z/m_{F}\Z.$

  Two compact symplectic toric orbifolds are equivariantly symplectomorphic if and only if their associated labeled polytopes are isomorphic. Moreover, every labeled polytope arises from some compact symplectic toric orbifold.
  \end{thm}

 The following lemma computes the labels of the facets of the moment polytope of
 $\CP(a^2,b^2,c^2)$, which plays an important role in our proof.

  \begin{lem}\label{lem:labels} Denote by $P$ the moment polytope of $\CP(a^2,b^2,c^2)$ of the associated residual $T^2$-action. Then the label $m_F$ on every facet $F$ of $P$ is 1.
  \end{lem}
  \begin{proof}
  Recall that the orbifold structure group of a point $[x:y:z]$ in the weighted projective plane $\CP(w_1,w_2,w_3)$ is $\Z / g\Z$ where $g$ is the greatest common divisor of those weights
  whose component is non-zero.
  For a point $[x:y:z]$ with $x \neq 0$ and $y \neq 0$, gcd of $a^2$ and $b^2$ is 1.
  Recall that $a, b, c$ are mutually coprime. (See \cite[Proposition 2.2]{Vi14}.)
  Therefore at those points the orbifold structure groups are trivial.
  Similarly every point $[x:y:z]$ but three orbifold points has a trivial orbifold structure group.
  Therefore each point fibering over the interior of each facets has a trivial orbifold structure group, which means the point is smooth.
  In other words, the labels of the other facets are 1. This finishes the proof.
  \end{proof}

Finally we find an explicit coordinate formula of the moment map $\phi_P: M_P \to (\R^2)^*$.
  Let $\pi:\phi_{K}^{-1}(0) \to M_P$ be the quotient map and $i: \phi_K^{-1}(0) \to \C^3$
  the inclusion map. Then by definition of the moment map \cite{MW74},
  the associated moment map $\phi_{P}:M_{P} \to (\R^{2})^{*}$ satisfies the following equation
  \begin{equation} \label{moment}
  \beta^{*} \circ \phi_{P} \circ \pi = \phi_{T^{3}} \circ i.
  \end{equation}

  Let $(x,y)$ (row vector) be the coordinate of $(\R^2)^*$ and $(z_1,z_2,z_3)$ (column vector)
  be the complex coordinate of $\C^3$.
  With respect to the standard basis, the map $\beta: \R^3 \to \R^2$ can be written as
  $$
   \beta = \begin{pmatrix} -(bl_{2}-1) & -(al_{1}-1) & 1 \\ -b^{2} & a^{2} & 0 \end{pmatrix}.
  $$
  Substituting $(z_1,z_2,z_3) \in \phi_{T^3}^{-1}(0)$ into the left hand side of \eqref{moment} and setting
  $$
  (x,y) = \phi_P(\pi(z_1,z_2,z_3)),
  $$
  we get
  \begin{eqnarray*}
  \begin{pmatrix} \frac{|z_{1}|^2}{2}-b^{2}c^{2} \\ \frac{|z_{2}|^2}{2} \\ \frac{|z_{3}|^2}{2} \end{pmatrix} = \beta^{*} \begin{pmatrix} x \\ y \end{pmatrix}
  = \begin{pmatrix} -(bl_{2}-1)x-b^{2}y \\ -(al_{1}-1)x+a^{2}y \\ x \end{pmatrix}
  \end{eqnarray*}

  By equating the first and the last terms of the equation and solving it for $(x,y)$,
  we obtain the coordinate formula of the associated moment map $\phi_{P}$ whose value at
  $[z_1:z_2:z_3] \in \CP(a^2,b^2,c^2)$ is given by
  \be\label{eq:phiP}
  (x,y) = \left (\frac{|z_3|^2(abc)^2}{a^2|z_1|^2+b^2|z_2|^2+c^2|z_3|^2},
  \frac{\left(|z_2|^2+(al_1-1)|z_3|^2\right)(bc)^2}{a^2|z_1|^2+b^2|z_2|^2+c^2|z_3|^2}\right)
  \ee
  for all $(z_1,z_2,z_3) \in \phi_{T^3}^{-1}(0) \subset \C^3$. This ends our symplectic description of
  symplectic orbifold $\CP(a^2,b^2,c^2)$, the $T^2$-action and its associated moment
  map $\phi_P: \CP(a^2,b^2,c^2) \to (\R^2)^*$.
  \begin{rmk}
  Note that $\CP(a^2,b^2,c^2)$ is covered by three orbifold charts.
  For example, $\{ [x:y:z] \mid x \neq 0\}$ is one of them.
  This is homeomorphic to $\C^2/\mu_{a^2}$ where the group action is given by $\zeta \cdot (y,z)=(\zeta^{b^{2}}y,\zeta^{c^{2}}z)$ i.e., its associated weights are given by $\frac{1}{a^2}(b^2,c^2)$.
  Denote this orbifold chart as $U_{a^2}$.
  Similarly, denote $U_{b^2}=\{ [x:y:z] \mid y \neq 0\}$ and $U_{c^2}=\{ [x:y:z] \mid z \neq 0\}$.
  It is easy to check that on each orbifold chart, the map $\phi_P$ in \eqref{eq:phiP} is invariant under the corresponding group action. Thus $\phi_P$ is well-defined.
  \end{rmk}

  Three orbifold points $[1:0:0],[0:1:0],[0:0:1]$ are mapped to the vertices opposite to the edges
  associated to $a^2u_1, b^2u_2, c^2u_3$, respectively.
  One can check that points of the form $[0:y:z] \in \CP(a^2,b^2,c^2)$ fibers over the
  points contained in the edge corresponding to $a^2u_1$.
  Similarly points of the form $[x:0:z],[x:y:0]$ fibers over points in the edge $b^2u_2,c^2u_3$, respectively.

We now visualize the image in $(\R^2)^*$ of the moment map $\phi_P$.
First we recall that every convex polytope $P$ can be written as the intersection of
 a finite number of oriented half-spaces. To define
 a labeled polytope, we attach a label $m_i$ on the $i$-th facet
 and consider the intersection
 $$
  P=\bigcap\limits^3_{i=1}\{ (x,y) \in (\R^2)^{*} : L_i(x,y)=\langle (x,y), m_{i}\mu_i \rangle - \lambda _i \geq 0 \}
 $$
 where $\mu_i$ is the $i$-th inward primitive integral vector normal to the $i$-th facet and $\lambda_i$ is areal number.

First we recall the polytope considered in \cite{Vi14}
for his construction of $T_{a,b,c}$.
In the above expression of general $P$, we consider the edge of affine length $a^2,b^2$, respectively, as the first, second facet. Then we get
 $$
  \mu_1 = (-(bl_{2}-1),-b^2), \, \mu_2 = (-(al_{2}-1),a^2),\, \mu_3 = (1,0)
 $$
 and consider inequalities
  \begin{eqnarray*}
  L_1(x,y)&:=&-(bl_{2}-1)x-b^{2}y+b^{2}c^{2} \geq 0\\
  L_2(x,y)&:=& -(al_{1}-1)x+a^{2}y \geq 0\\
  L_3(x,y)&:=& x \geq 0
  \end{eqnarray*}
  with $\lambda_1=-b^{2}c^{2}, \lambda_2=\lambda_3=0$. We denote the
  polytope given by this equation by
  \be\label{eq:P}
  P = P(a^2,b^2,c^2).
  \ee
This is precisely the one used in \cite{Vi14}.
See Figure \ref{fig:2}.
  \begin{figure}[h]
  \centering
  \includegraphics[scale=1.1]{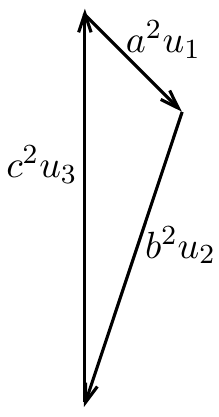}
  \caption{Moment polytope of $\CP(a^2,b^2,c^2)$}
  \label{fig:2}
  \end{figure} \\

 The following proves that the
 \emph{labeled polytope $P$} of the toric orbifold $\CP(a^2,b^2,c^2)$ is exactly the moment polytope described above equipped with label 1 on the interior of each facet.
\begin{prop} The moment image of $\phi_P:\CP(a^2,b^2,c^2) \to (\R^2)^*$ given above is
\eqref{eq:P}.
\end{prop}
\begin{proof} It is straightforward to check from the explicit formula \eqref{eq:phiP}.
\end{proof}

As described in \cite{Vi14}, its moment polytope under the residual torus action is
  the triangle with edges parallel to the vectors $a^2u_1, \, b^2u_2, \, c^2u_3$ where they satisfy the balancing condition
  $$
  a^2u_1+b^2u_2+c^2u_3=(0,0)
  $$
  and have the form
  $$
  u_1=(b^2,-(bl_2-1)), \, u_2=-(a^2,al_1-1),\, u_3=(0,1)
  $$
  and with the vertices located at the projection of the orbifold points.
  From $a^2u_1+b^2u_2+c^2u_3=(0,0),$ we obtain a relation
  \be \label{eq:al2+bl1=3c}
  al_2+bl_1=3c.
  \ee

  Here each of $l_1, \, l_2, \, l_3$ is a positive integer coprime to $a, \, b, \, c$ respectively.
  Each of
  these integers can be realized as a winding number of some section of a trivialization over the boundary of unit disk appearing in the definition of \emph{Lagrangian pinwheel},
  embedded in $\CP^2$. (See \cite{Kho13,ES18}.) Furthermore, Evans and Smith \cite{ES18} proved that
  the integers satisfy the following congruence relations;

  \begin{align*}
  l_{1}^2 &\equiv -9 \pmod{a} \\
  l_{2}^2 &\equiv -9 \pmod{b} \\
  bl_{1} &\equiv \pm 3c \pmod{a} \\
  cl_{1} &\equiv \pm 3b \pmod{a} \\
  al_{2} &\equiv \pm 3c \pmod{b} \\
  cl_{2} &\equiv \pm 3a \pmod{b} \\
  \end{align*}

  Then using the coordinates of $(\R^2)^*$, $l_{i}$'s could be computed explicitly from Figure \ref{fig:4} below, and following a series of nodal surgery operations followed by transferring the cut operation, starting from the moment polytope of $\CP^2$ which corresponds to Markov triple $(1,1,1)$. (See \cite[Section 2]{Vi14}.)

  \begin{rmk}
  When deriving \eqref{eq:phiP}, using $-(bl_2-1)x-b^2y=\frac{|z_1|^2}{2}-b^2c^2,$ we get another expression for
  $$
  y=c^2 - \frac{(|z_1|^2+(bl_2-1)|z_3|^2)(ac)^2}{a^2|z_1|^2+b^2|z_2|^2+c^2|z_3|^2}.
  $$
  Then one can easily check that using coordinates on orbifold charts $U_{a^2}, \, U_{b^2}, \, U_{c^2}$ the image of $\phi_P$ is exactly $P$.
  In other words, on $U_{a^2}$, its image point $(x,y)$ of \eqref{eq:phiP} satisfies
  $$
  L_1(x,y) > 0, \, L_2(x,y) \geq 0, \, L_3(x,y) \geq 0.
  $$
  On $U_{b^2}$, the image point $(x,y)$ satisfies
  $$
  L_1(x,y) \geq 0, \, L_2(x,y) > 0, \, L_3(x,y) \geq 0.
  $$
  On $U_{c^2}$, the image point satisfies
  $$
  L_1(x,y) \geq 0, \, L_2(x,y) \geq 0, \, L_3(x,y) > 0.
  $$
  \end{rmk}

  The associated symplectic form $\omega_P$ can be written explicitly in terms of labelled polytope data
  using the analogue of Guillemin's formula \cite{Gu96}.

  \begin{thm}\emph{\cite{CDG03}}\label{GulleminFormula}
  Let $P$ be a labelled polytope as above and $d$ be the number of facets of $P$.
  Define a function
  \begin{eqnarray}
  L_\infty(u) = \sum_{i=1}^{d}\langle u, m_i \mu_i \rangle \label{eq:ellinfty}
  \end{eqnarray}
  on $(\R^n)^*$.
  We have
  \begin{equation}\label{eq:omegaP}
  \omega_P =\sqrt{-1}\partial\overline{\partial}\phi_{P}^*\left(\sum_{i=1}^{d}\lambda_i\log L_i + L_\infty\right)
  \end{equation}
  on $\phi_{P}^{-1}(\text{\rm Int}\,P)$.
  \end{thm}

  In our case, $d=3$ and $n=2$.
  For $(X,Y)$ in $(\R^{2})^*$, $$L_\infty(X,Y)=\langle (X,Y), \sum_{i=1}^{3} \mu_{i} \rangle = [3-(al_{1}+bl_{2})]X + (a^{2}-b^{2})Y.$$

\subsection{Symplectic rational blow-down and almost toric fibration}

One realization of
  the $T_{a,b,c}$ tori is given as the barycentric fibers of various
  almost toric fibrations on $\CP^2$ depending on $(a,b,c)$.
  As mentioned before, its base diagram $(B,\Lambda)$ looks like the moment polytope $P=P(a^2,b^2,c^2)$
  of $\CP(a^2,b^2,c^2)$ except that there is a node at each vertex with a cut in the direction of the corresponding node.

  Each small neighborhood of a vertex can be realized as an almost toric base with one node whose boundary is a lens space of the form $L(k^2,kl-1).$
  Let us describe a smooth 4 dimensional manifold fibering on this neighborhood, the corresponding base in $P$ and its symplectic rational blow-down surgery.

  Let $k,l$ be a pair of coprime positive integers.
  Consider the base diagram given by the open subset $U_{k,l} \subset (\R^2)^*$ consisting of the points
  $(x,y)$ contained in the intersection of the half-spaces
  $$
  y \geq \frac{kl-1}{k^2}x, \, x \geq 0, \, y > 0,
  $$
  which is bounded by an arbitrary embedded arc $\gamma$ joining two points on the edges respectively contained  in
  the two lines $y = \frac{kl-1}{k^2}x$, $x = 0$. It also carries a node on a cut in the direction of the vector $(k,l)$.

  \begin{figure}[h]
  \centering
  \includegraphics[scale=0.9]{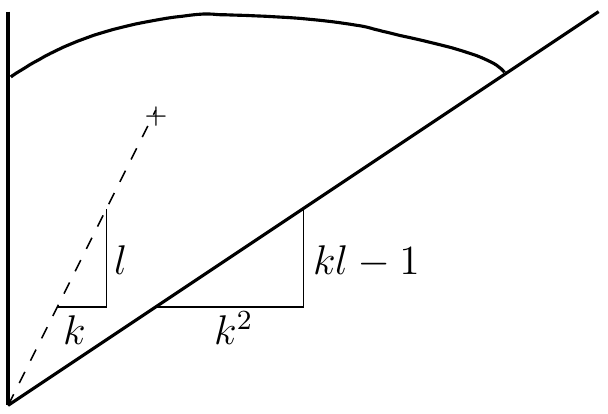}
  \caption{Almost toric base $U_{k,l}$}
  \label{fig:3}
  \end{figure}

  A smooth 4 dimensional manifold $B_{k,l}$ that fibers over $U_{k,l}$
 is a \emph{rational homology ball} whose boundary is a lens space $L(k^2,kl-1).$
  This lens space fibers over the embedded arc $\gamma$ \cite[Section 9.3]{Sy03} contained in $U_{k,\ell}$.

  On the other hand, it is proved in \cite[Proposition 2.2]{Vi14} that the boundary of a neighborhood of
  the orbifold point projected to the vertex opposite to $a^{2}u_1, b^2u_2,c^2u_3$ is the lens space of the form $$
  L(a^2,al_{1}-1),\, L(b^2,bl_{2}-1), \, L(c^2,cl_{3}-1),
  $$
  respectively.
  Here each $l_i$ for $i=1,2,3$ is a positive integer coprime to $a,b,c$, respectively.
  Now we consider a small neighborhood $N_{a^2}$ of an orbifold point which fibers over the vertex
  opposite to the edge $a^2u_1.$
  As two collar neighborhoods of a boundary of $N_a$ and of $B_{a,l_1}$ fiber over the same simply connected base, there exists a symplectomorphism $\psi_a$ from a collar neighborhood of the boundary of $N_{a^2}$ to a collar neighborhood of the boundary of $B_{a,l_1}$.
  Similar results hold for the remaining neighborhoods $N_{b^2}$ and $N_{c^2}$.(See \cite{Sy01})

  Then the rational blow-down surgery replaces a neighborhood $N_{a^2},N_{b^2},N_{c^2}$ of an orbifold point by a rational homology ball $B_{a,l_1},B_{b,l_2},B_{c,l_3}$, respectively.
  Furthermore, the symplectic structures $\omega_P$ and $\omega_{a,l_1},\omega_{b,l_2},\omega_{c,l_3}$ induce a symplectic structure as follows.

  As performed in \cite[Corollary 2.5]{Vi14}, applying rational blow-down on each neighborhood of the
  three orbifold points of $\CP(a^2,b^2,c^2)$ yields an almost toric fibration of $\CP^2$. We denote by $\CP^2_{a,b,c}$
   the total space of this almost toric fibration for some base $(B,\Lambda)$. More precisely,
applying the rational blow-down surgery on each neighborhood of three orbifold points in $\CP(a^2,b^2,c^2)$, we obtain an almost toric manifold $\CP^2_{a,b,c}$ symplectomorphic to $\CP^2$
  $$
  \CP^2_{a,b,c}: = \left(\CP(a^2,b^2,c^2) \setminus (N_{a^2} \cup N_{b^2} \cup N_{c^2})\right) \bigcup _{\psi} (B_{a,l_1} \cup B_{b,l_2} \cup B_{c,l_3})
  $$
  equipped with an interpolated symplectic form
  \be\label{eq:tilde-omega}
  \widetilde{\omega}=(1- \chi_a - \chi_b-\chi_c) \omega_{P} + \chi_a \lambda_a \omega_{a,l_1} + \chi_b \lambda_b \omega_{b,l_2} + \chi_c \lambda_c \omega_{c,l_3}.
  \ee
  Here each $\chi_a$ is a cut-off function which is $1$ on the corresponding rational homology ball and $0$ otherwise. And $\lambda_a, \lambda_b, \lambda_c$ are suitably chosen positive real-valued functions
  so that $\widetilde \omega$ becomes nondegenerate and closed.

  Corresponding to base diagram $(B,\Lambda)$,
  $P$ has three nodes with eigenrays \cite{Vi14}) pointing towards
  $$
  (a,l_1), \, (b,-l_2), \, (-(a^2+b^2),bl_2-al_1)
  $$
  issued at the vertices opposite to $a^2u_1, b^2u_2, c^2u_3$, respectively.
  The three eigenrays meet at one point, the \emph{labeled barycenter},
  in the interior of $P$.
  (See Figure \ref{fig:4}. Red dashed lines represent respective cuts and the blue dot represents the labeled barycenter.)

  Following \cite[Remark 2.3]{Vi14}, we set $w_1=-(a,l_1), \, w_2=(-b,l_2)$ and $w_3$ as the
  vectors representing cuts.
  They satisfy relations
  \be \label{eq:wbary-u_3}
  \frac{acw_2-bcw_1}{3}=c^2u_3
  \ee

  \be \label{eq:wbary-u_2}
  \frac{bcw_1-abw_3}{3}=b^2u_2
  \ee

  \be \label{eq:wbary-u_1}
  \frac{abw_3-acw_2}{3}=a^2u_1
  \ee

  and
  \be \label{eq:wbary}
  \frac{abc}{3}(aw_1+bw_2+cw_3)=0.
  \ee
  From \eqref{eq:wbary-u_2}, we have $aw_3=cw_1-3bu_2=(a(3ab-c),3b(al_1-1)-cl_1)$.
  The slope of $w_3$ is then
  \beastar
  \frac{l_1(3ab-c)-3b}{a(3ab-c)} &=& \frac{l_1}{a} - \frac{3b}{a\frac{a^2+b^2}{c}} \\
  &=& \frac{l_1}{a} - \frac{3bc}{a(a^2+b^2)}.
  \eeastar
  By \eqref{eq:wbary-u_1}, the slope of $w_3$ is
  \be
  \frac{(c-3ab)l_2+3a}{b(3ab-c)}.
  \ee
  Also from \eqref{eq:wbary}, the slope of $w_3$ can be written as
  \be
  \frac{al_1-bl_2}{a^2+b^2}.
  \ee
  These three expressions of the slope of $w_3$ are indeed the same.

  By translation, we may identify the lower left vertex of $P$ with the origin in $(\R^2)^*.$
  Consequently vertices opposite to edges $b^2u_2,c^2u_3$ respectively are located at the points
  $$
  (0,c^2), \, (a^2b^2,c^2-a^2(bl_2-1)).
  $$
  Then we derive $(a^2b^2,c^2-a^2(bl_2-1))=(a^2b^2,b^2(al_1-1))$ using \eqref{eq:al2+bl1=3c}.

  \begin{figure}[h]
  \centering
  \includegraphics[scale=1.1]{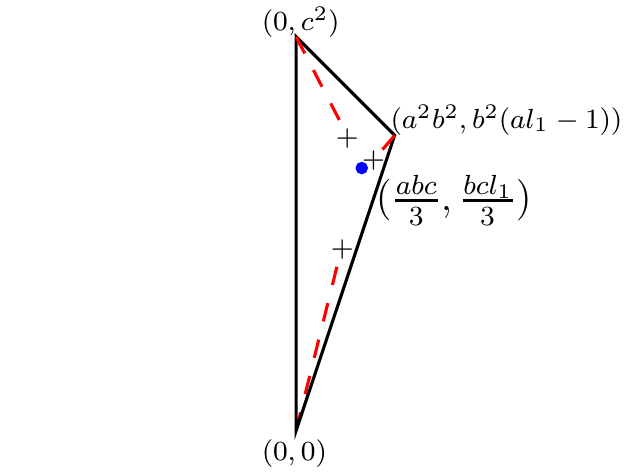}
  \caption{The polytope $P$}
  \label{fig:4}
  \end{figure}

  We summarize the above discussion into

  \begin{lem}\label{lem:vianna-tori-location} Let $P$ be the polytope associated to
  the above base diagram $(B,\Lambda)$. Then the torus $T(a^2,b^2,c^2)$ is located at
  the point
  \be\label{eq:bary-center}
  \left(\frac{abc}{3},\frac{bcl_1}{3}\right) \in \Int P(a^2,b^2,c^2).
  \ee
  \end{lem}
  \begin{proof} Using the slope formula we derived above,
  we easily check that the three eigenrays are given by
  \beastar
  y & = & \frac{l_1}{a}x, \\
  y& = & c^2-\frac{l_2}{b}x,\\
   y& =& b^2(al_1-1)+\bigg[\frac{l_1}{a} - \frac{3bc}{a(a^2+b^2)}\bigg](x-a^2b^2)
  \eeastar
  and that they meet at one point
  $$
  \left(\frac{abc}{3},\frac{bcl_1}{3}\right).
  $$
  \end{proof}

\subsection{Normalization of the polytope}

  When we choose $\chi_i$'s and $\lambda_j$'s in \eqref{eq:tilde-omega},
  we require the (volume) normalization condition in addition so that
  $$
  \int_{\CP^2_{a,b,c}} \widetilde{\omega}^2 = \int _{\CP^2} \omega_{\text{\rm FS}}^2.
  $$
  Such a choice can be always made by suitably choosing the functions $\lambda_i$,
  and then Moser's deformation trick produces a diffeomorphism
  between the two symplectic forms $\widetilde \omega$ and $\omega_{\text{\rm FS}}$
  $$
  \alpha:(\CP^2, \omega_{\text{\rm FS}}) \to (\CP^2_{a,b,c}, \widetilde{\omega})
  $$
  such that $\alpha^*\widetilde{\omega} = \omega_{\text{\rm FS}}$. We fix such a symplectic
  (actually K\"ahler) form $\widetilde \omega$.
  (See \cite{Vi13} for more detailed explanation.)

To apply the above mentioned Moser's deformation
to the pair of forms $\alpha^*\widetilde \omega$ and $\omega_{\text{\rm FS}}$, we need to suitably normalize
the size of the polytope $P(a^2,b^2,c^2)$ so that the cohomology classes
$[\alpha^*(\widetilde \omega)],\,  [\omega_{\text{\rm FS}}] \in H^2(\CP^2,\R)$ should be the same.
Since $\dim_\R H^2(\CP^2,\R) = 1$, we know $[\alpha^*(\widetilde \omega)] = C [\omega_{\text{\rm FS}}]$
for some positive constant $C > 0$.

We now determine what this $C$ is. We first recall that monotonicity constants of
all monotone Lagrangian tori in $\CP^2$ are the same. For example, the Maslov index $2$ discs
of $T_{a,b,c}$ in $\CP^2$ have the same symplectic areas independent of $a, \, b,\, c$.

On the other hand, by the classification theorem from \cite{CO06},
  there exist three obviously seen Maslov index 2 holomorphic discs attached to $T(a^2,b^2,c^2)$
  which are associated to the facets of the polytope $P$.
  Denote these holomorphic discs by $D(\mu_i)$ for $i=1,2,3$ corresponding to the $i$-th facet of $P.$

  The following fact is stated in \cite[Paragraph after Prop. 2.4]{Vi14} without proof.
  Because this is an important element in our study of the present paper,
  we give its proof for readers' convenience.

  \begin{prop}\label{prop:scaling} Consider the polytope $P = P(a^2,b^2,c^2)$ described around \eqref{eq:P}.
  If we scale $P$ by dividing by $abc$, then the $\alpha^*(\widetilde \omega)$-symplectic areas of Maslov index 2 holomorphic discs
  attached to $T_{a,b,c}$ are the same as that of $\omega_{\text{\rm FS}}$. In particular, $[\alpha^*(\widetilde \omega)] = [\omega_{\text{\rm FS}}]$
  for all $a, \, b,\, c$.
  \end{prop}
  \begin{proof}
  We start with the following area formula for holomorphic disc of Maslov index 2 from \cite{CO06}
  for general toric manifolds.

  \begin{lem}[Theorem 8.1 \cite{CO06}]
  Let $\omega_X$ be the toric symplectic form, which is the reduced form
  of the standard symplectic form $\omega_0$ on $\C^N$ under the linear sigma model construction.
  Consider the residual $T^n$ action on $X$ and its moment map $\pi: X \to {\frak t}^* \cong \R^n$.
  Let $L = \pi^{-1}(A)$ be the fiber torus based at $A \in \Int P$. Then
  the symplectic area of the holomorphic disc corresponding to the i-th facet is given by
  $
  2\pi(\langle A,\mu_i \rangle-\lambda_i).
  $
  \end{lem}

  Using this area formula, we compute
  \begin{eqnarray*}
  \text{area}\ D(\mu_1)&=& 2\pi (\langle A,\mu_1 \rangle + b^2 c^2) \\
  &=& 2\pi \left(-\frac{abc}{3}(bl_{2}-1)-\frac{bcl_{1}}{3}b^{2} + b^2 c^2\right) \\
  &=& \frac{2\pi}{3}bc\left(3bc-a(bl_{2}-1)-b^{2}l_{1}\right) \\
  &=& \frac{2\pi}{3}bc\left(3bc+a-b(bl_{1}+al_{2})\right)
  \end{eqnarray*}
  By \eqref{eq:al2+bl1=3c}, we compute
  \begin{eqnarray*}
  \text{area}\ D(\mu_1) = \frac{2\pi}{3}abc.
  \end{eqnarray*}
  Similarly, for the second facet,
  \begin{eqnarray*}
  \text{area}\ D(\mu_2)&=& 2\pi (\langle A,\mu_2 \rangle ) \\
  &=& 2\pi \left(-\frac{abc}{3}(al_{1}-1)+\frac{bcl_{1}}{3}a^{2} \right) \\
  &=& \frac{2\pi}{3}abc.
  \end{eqnarray*}
  For the third facet,
  \begin{eqnarray*}
  \text{area}\ D(\mu_3) = 2\pi (\langle A,\mu_3 \rangle ) = \frac{2\pi}{3}abc.
  \end{eqnarray*}
  (This explanation shows that every holomorphic disc has the same area
  as it should be because
  because Lagrangian tori $T(a^2,b^2,c^2)$ are monotone.)
  Therefore after if we scale the symplectic form $\widetilde\omega$ by $\frac{1}{abc}$,
  all holomorphic discs of Maslov index 2 has area $\frac{2\pi}{3}$ which is the same as
  that of Clifford torus of in $\CP^2$ with respect to $\omega_{\text{\rm FS}}$.
  Obviously this area is independent of the choice of Markov triple $(a,b,c)$.
  \end{proof}

\section{Lower bound for the relative Gromov area}
  Let $(M^{2n},\omega)$ be a symplectic manifold. We recall the definition of \emph{Gromov area}.
  Denote
  $$
  B^{2n}(d) =\{z \in \C^{n} : |z| \leq d \}
  $$
  for $d > 0$.

  \begin{defn}[Relative Gromov area] Let $L \subset M$ be a compact subset.
  Consider a symplectic embedding $e: B^{2n}(d)\to M \setminus L$. The \emph{relative Gromov area} is defined by
  $$
  c_G(M; L): = \sup_e\{ \pi d^2 \mid e:B^{2n}(d)\to M \setminus L \text{ is a symplectic embedding}\}
  $$
  \end{defn}

  We are interested in studying the behavior $c_G(\CP^2; T_{a,b,c})$ as $\max\{a, b, c\} \to \infty$.
  For this purpose, we first recall two methods of finding a symplectic balls in the context of
  toric manifolds.

  The first one is given by Karshon \cite{Ka94} who uses the shape of triangle and the other is
  given by  Mandini and Pabiniak \cite{MP18} who uses the shape of diamond in the
  moment polytope.

\subsection{Almost toric blowup and symplectic balls}
\label{subsec:almost-toric-blowup}

In this subsection, we follow the exposition given in \cite[Section~2.4]{Vi16a} on the almost
 toric blowup to which the readers are referred for details. See also
 \cite{Zu03,Sy03,SyLe10}.

In short, the picture below describes an almost toric blowup. The exceptional
 curve lies over the dashed line, consisting of one circle on each fibre
 that collapses as it approaches the edge and as it approaches the node.

 \begin{figure}[h!]

\begin{center}

\centerline{\includegraphics[scale=0.35]{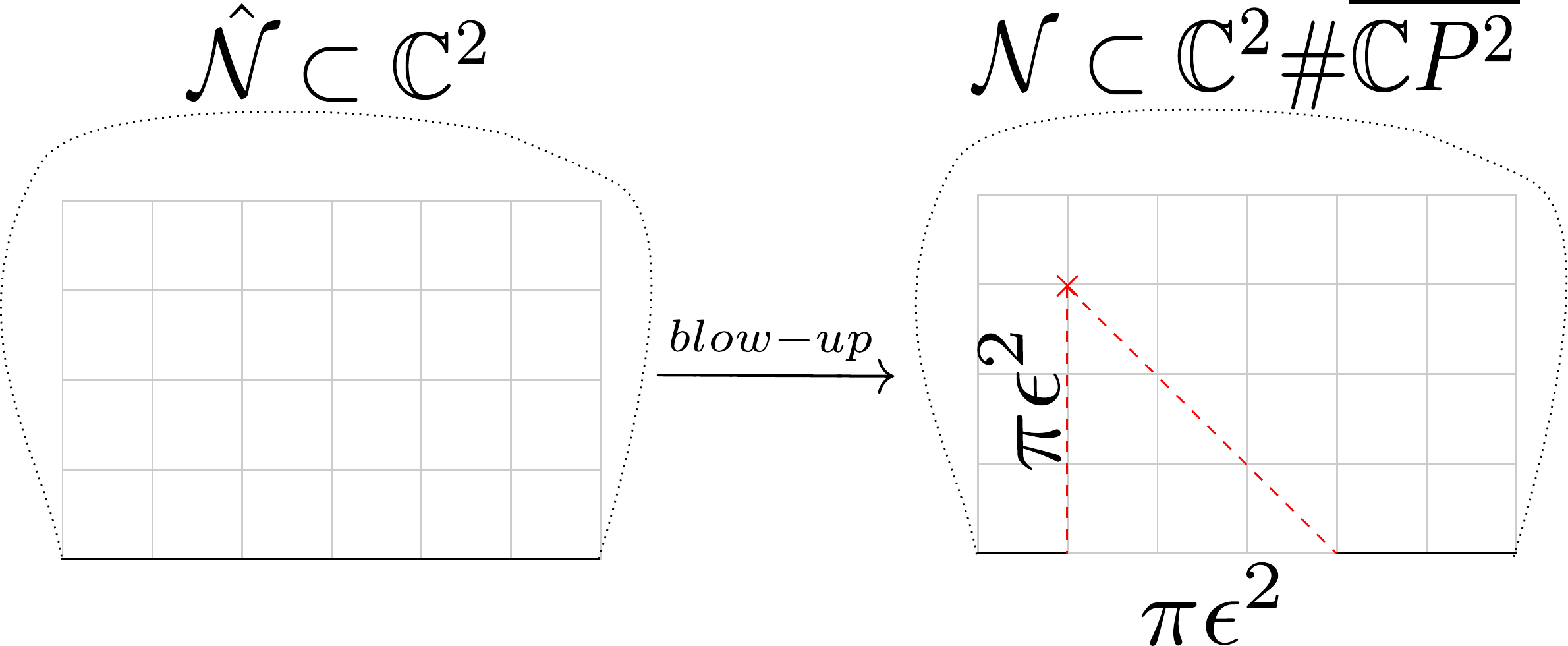}}

\caption{Almost toric blowup.}
\label{fig: ATBlup1}

\end{center}
\end{figure}

In particular, one must be able to find a symplectic ball of Gromov area $\pi
\epsilon^2$, in a neighbourhood of an affine triangle in an
ATF (almost toric fibration), corresponding to the missing triangle after
the almost toric blowup.

More precisely, suppose we see a triangle inside a toric region of an ATF, as
illustrated in Figure \ref{fig: Triang}. Viewing the preimage of any small
neighborhood of this triangle as a blowdown of the
corresponding neighborhood of the exceptional curve as in Figure \ref{fig:
ATBlup1}, we can then infer that
there is a symplectic ball of capacity $\pi \epsilon^2$ projecting
into the preimage of this neighborhood of the triangle.

\begin{figure}[h!]

\begin{center}

\centerline{\includegraphics[scale=0.35]{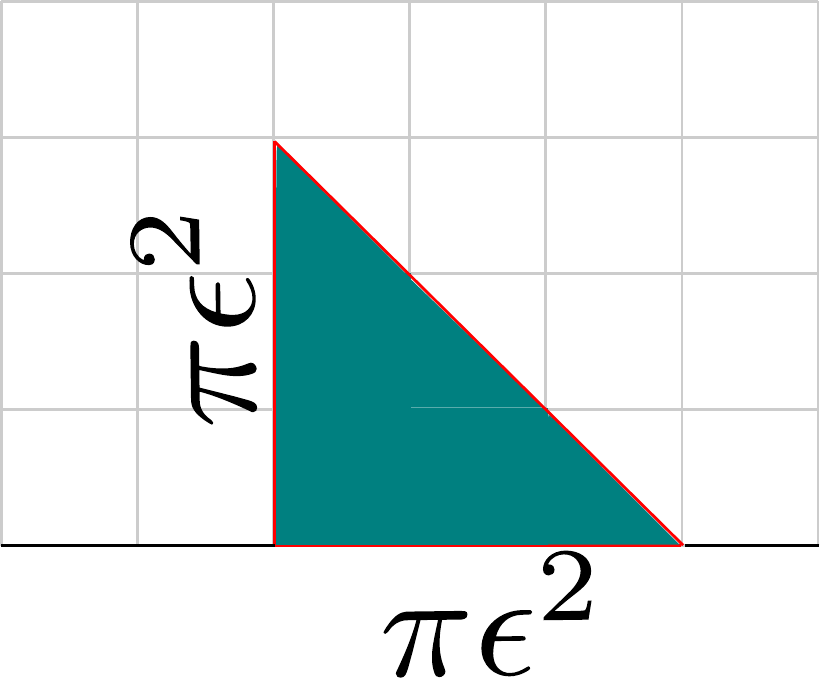}}

\caption{A symplectic ball $B$ of capacity smaller than $\pi \epsilon^2$
projects into the shaded triangle.}
\label{fig: Triang}

\end{center}
\end{figure}

So for our purpose, every time we encounter a triangle as in Figure \ref{fig:
Triang}, we know there is a symplectic ball $B$ of any given capacity smaller
than $\pi \epsilon^2$ that projects into the shaded triangle.

Just for visualization purpose, we describe how a symplectic ball $B$ centred at
a point lying over an edge of an ATF, and projecting into the triangle of Figure
\ref{fig: Triang}, should look like.

We first describe how its boundary $\del B \cong S^3$ intersect
each fibre of the triangle. To ensure that the topological sphere
described below is indeed the boundary of a symplectic four ball, one needs to make more
specific choices for these intersections. We omit these details,
since by the previous discussion, we do not really need them.

The boundary $\del B \cong S^3$
intersects the corresponding triangle in the toric sector of the ATF as follows:

 \begin{enumerate}[label=\alph*)]

  \item Consider the edge of the triangle corresponding to the
  intersection with edge of the ATF. The circle fibres,
  corresponding to the relative interior of the edge, intersect
  $\del B$ in two points. The circles fibering over the endpoints of
  the edges, intersect $\del B$ in one point each. Hence, $\del B$
  intersects the fibre over the edge in a circle.

  \item The tori living over the remaining edges of the triangle intersect $\del
  B$ in one circle. The class of this circle is the collapsing class
  associated to the edge. So the circles collapse to a point as we approach the
  edge, which is consistent with the previous item a).

  \item The tori living over the interior edge, intersect $\del B$ in two
  circles, also in the collapsing class associated to the edge. Hence, they
  collapse to two points as we approach the edge of the item a). If we approach
  the fibres of described in item b), these two circles collapse to the corresponding
  single circle.

 \end{enumerate}

In particular, consider a segment in the triangle, parallel to the edge of the
ATF, connecting points of the edge of the triangle. It divides the triangle in
two parts, the top part being a similar triangle. The intersection of $\del B$
with the fibres over this segment form a torus. The top part then intersects
$\del B$ in a solid
torus, where the family of tori living over the corresponding parallel edges
collapse to the circle living over the vertex. The bottom part also intersects
$\del B$ in a solid torus, now the corresponding parallel segments converge
to the edge of the intersecting the edge of the ATF, whose fibres intersect
$\del B$ in a circle, as in item a) above. It is easy to see that this circles
correspond to generators of $H_1$ of the torus living over our initial segment.
Hence, we have indeed $\del B \cong S^3$.

To get $B$, we consider $S^3$'s as above, projecting to smaller triangles
similar to each other, eventually collapsing to a point in the middle of the
edge of the starting triangle.

One is able to see that we can get balls of any capacity smaller than
$\pi \epsilon^2$, projecting inside our given triangle. In particular,
we can get a symplectic ball of capacity $\pi \epsilon^2$, if we get
this triangle inside a slightly bigger one in our ATF.

\subsection{Method by Mandini and Pabiniak \cite{MP18}}

We first recall a result by Mandini and Pabiniak \cite{MP18}. Following \cite{MP18},
  we consider the subset
  $$
  \Diamond(d)=\{ (x,y)\in \R^2 : |x|+|y| < \frac{d}{2} \}.
  $$

  \begin{prop} \emph{\cite{MP18}} \label{prop:DiamondToric}
  For each $\varepsilon > 0$ the 4-ball $B^4(\sqrt{2(d-\varepsilon)})$ of capacity $2\pi(d-\varepsilon)$ symplectically embeds into $\Diamond(d) \times (0,2\pi)^2 \subset \R^2 \times \R^2$.
  Therefore, if for a toric manifold $(M^4,\omega)$ with moment map $\phi$, $$\Psi(\Diamond(d))+x \subset \text{\rm Int}\, \phi(M)$$ for some $\Psi \in GL(2,\Z)$ and $x \in \R^2$, then the Gromov area of $(M^4,\omega)$ is at least $2 \pi d$.
  \end{prop}

Construction of such a symplectic embedding of a 4-ball $B^4(\sqrt{2(d-\varepsilon)})
  $ into $\Diamond(d) \times (0,2\pi)^2 \subset \R^2 \times \R^2$ is given in the proof of
  \cite[Lemma 4.1]{LMS13} \emph{using
  \cite[Lemma 3.1.8]{Sch_book} which is irrelevant to the toric structure of a symplectic manifold.}
  They use only an area-preserving map from a 2-ball of capacity $d$ to rectangle $R(d)=(0,d)\times(0,1)$, sending concentric circles to loops, rectangles with four corners smoothed such that the area enclosed by each smoothed rectangle in $R(d)$ is equal to the area enclosed by a concentric circle in the 2-ball, in $R(d)$.

  Let $(M^{4},\omega)$ be a toric symplectic manifold.
  If an affine transformation maps a diamond $\Diamond(d)$ into the interior of $\phi(M),$ then some subset of $\phi^{-1}(\Diamond(d))$ is symplectomorphic to $\Diamond(d) \times (0,2\pi)^{2}.$
  Here we use the identification $S^{1}=\R/2\pi\Z.$
  Then the above symplectic embedding of the diamond induces a symplectic embedding of $B^4(\sqrt{2(d-\varepsilon)})$ into $(M^{4},\omega).$

  Adopting the same idea in the almost toric case,
  not only for toric manifolds but also for almost toric manifolds, Proposition \ref{prop:DiamondToric} will hold
  the case with suitable modifications.

  \begin{prop} \label{prop:DiamondAlmostToric}
  Let $(M,\omega)$ be an almost toric manifold with almost toric fibration $\pi: (M,\omega) \to B.$
  If $\Psi(\Diamond(d))+x \subset \text{\rm Int}\, \pi(M) \setminus \{ \emph{nodes} \}$ for some $\Psi \in GL(2,\Z)$ and $x \in \R^2$, then the Gromov area of $(M^4,\omega)$ is at least $2 \pi d.$
  \end{prop}

  \begin{proof}
  A crucial ingredient in the toric case exploits the fact that each fiber over an interior point of its moment polytope is a 2-torus and we have $\phi^{-1} (\Diamond(d))\cong\Diamond(d) \times T^2$ symplectically via the action-angle variables.
  Let $\pi: (M,\omega) \to B$ be an almost toric fibration.
  In the almost toric base a smooth fiber over its interior point is a 2-torus away from
  the singular values of $\pi$.

  Assume that there is an affine transformation $\Phi$ of $\Diamond(d)$ mapping into $\text{\rm Int}\, \pi(M) \setminus \{ \text{nodes} \}.$
  Then some subset of $\pi^{-1}(\Phi(\Diamond(d))) \subset M \setminus \{ \text{\rm nodal fibers}\}$ is symplectomorphic to $\Diamond(d) \times (0,2\pi)^{2}.$
  Symplectic embedding of the 4-ball $B^4(\sqrt{2(d-\varepsilon)})$ of capacity $2\pi(d-\varepsilon)$ into $\Diamond(d) \times (0,2\pi)^2 \subset \R^2 \times \R^2$ in \cite{LMS13} completes the proof.
  \end{proof}

  \begin{rmk}
  Recall that both nodal trade and nodal slide operations induce two diffeomorphic smooth 4-manifolds with isotopic symplectic forms.
  Performing nodal slide of a node towards either the vertex or the barycenter of a base diagram along each eigenray allow us to find the size ``$d$'' of a diamond $\Diamond(d)$ while avoiding all nodes inside the base diagram. Since Gromov area is invariant under a symplectomorphism, combining these,
 we can find a maximal lower bound for the Gromov area of an almost toric manifold.
  \end{rmk}

\section{Symplectic balls in the complement of $T_{a,b,c}$ in $\CP^2$.}
\label{sec:balls}

By the discussion given in Subsection \ref{subsec:almost-toric-blowup},
in order to see a symplectic ball in the complement of $T_{a,b,c}$, it is enough
to identify a corresponding triangle with one side in the edge of the ATF base diagram, or a
diamond in the interior of the base avoiding $T_{a,b,c}$.

For the simplicity of the constants appearing in this section,
we scale the symplectic form so that the area of the Maslov index 2 disk
is $1$ or the area of the line is $3$.
Therefore, each side of the toric diagram has length 3, corresponding to the area of the
line and the area of the anti-canonical divisor is 9, since it has degree 3.
Therefore, the total area of the boundary of any ATF base diagram is 9. Moreover looking
at the orbifold limit the largest edge has length $\ge 3$.

We call a \emph{monotone triangle}, a triangle with one edge at the boundary of the
base diagram of the ATF, with all the affine lengths of the edges being $1$, and
the associated symplectic ball a \emph{monotone ball}.
Inside a neighbourhood of this triangle projects a ball of capacity one
that endows the monotone symplectic form in $\BlI$ after blown down.

 \begin{figure}[h!]

\begin{center}

\centerline{\includegraphics[scale=0.7]{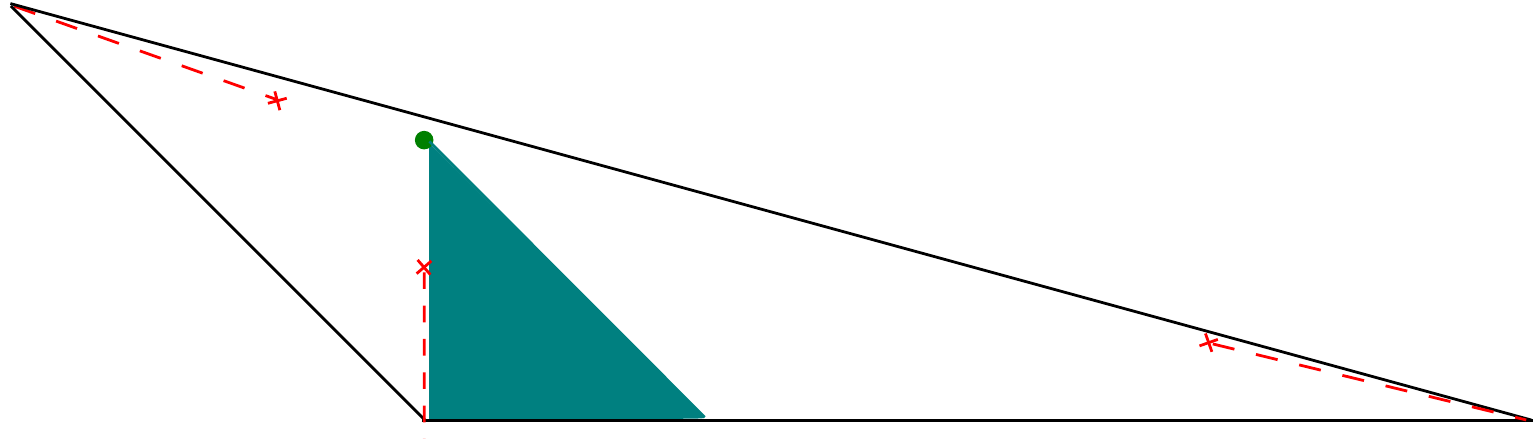}}

\caption{A symplectic ball in the complement of the monotone fibre in an ATF of $\CP^2$ projects
into a arbitrarily small neighbourhood of the monotone triangle (shaded).
}
\label{fig:abc_ATF}

\end{center}
\end{figure}

\begin{thm} Rescale the standard Fubini-Study form so that the area of the line is $2\pi$ so that
the Maslov index 2 disk has area $2\pi/3$. Then

 \[ \inf_{(a,b,c) \in \mathfrak M} c_G(\CP^2; T_{a,b,c}) \ge \frac{2\pi}{3}
  \]
\end{thm}
\begin{proof}
Up to $SL(2,\Z)$, we can always take the largest edge of ATF base diagram associated to
each Markov triple $(a,b,c)$ to be horizontal, with one
of the cuts intersecting that edge vertical, as illustrated in Figure \ref{fig:abc_ATF}.
Hence the monotone fibre $T_{a,b,c}$ is at height 1. Therefore we can find
a equilateral right triangle so that one of the vertices thereof is put
right at the point.
Because the affine length of the horizontal edge is $ \ge 3$ (in
particular grater than 1), we can always find a neighbourhood of a monotone
triangle in the complement of the monotone fibre $T_{a,b,c}$. This finishes the proof.
\end{proof}

 Even though our construction indicates the lower bound given in the above theorem
 may be optimal, it is not clear whether it is indeed the case. In fact, as
 already mentioned in the introduction, the lower bound will be bigger for
 an individual torus, since we can get a neighbourhood of the monotone triangle
 -- see Figure \ref{fig:abc_ATF} again. In other words,
 $c_G(\CP^2; T_{a,b,c}) \ge \frac{2\pi}{3} + \varepsilon_{a,b,c}$, for some
 $\varepsilon_{a,b,c} > 0$.
As we mentioned in the introduction, Biran-Cornea showed $c_G(\CP^2;T_{text{\rm Cl}}) = \frac{4\pi}3$
for the Clifford torus $T_{\text{\rm Cl}} \cong T_{1,1,1}$ in \cite{BC09A}.

However we conjecture the above lower bound is indeed optimal.

 \begin{conj} \label{conj:Capacity}
   There is no monotone ball in the complement of $T_{a,b,c}$,
   and \[ \inf_{(a,b,c) \in \frak M} c_G(\CP^2; T_{a,b,c}) =
    \frac{2\pi}{3}\]

   with the convention of $\frac{2\pi}{3}$ being the capacity of the monotone
   ball.
 \end{conj}

We note that the ball presented in this theorem intersects the `boundary divisor'
by construction. Denote by $E$ the preimage of the boundary of the base diagram.
So we consider $\CP^2 \setminus E$, the complement of $E$ (still endowed
with the finite volume form coming from $\CP^2$, i. e., without completing it).

 \begin{figure}[h!]

\begin{center}

\centerline{\includegraphics[scale=0.7]{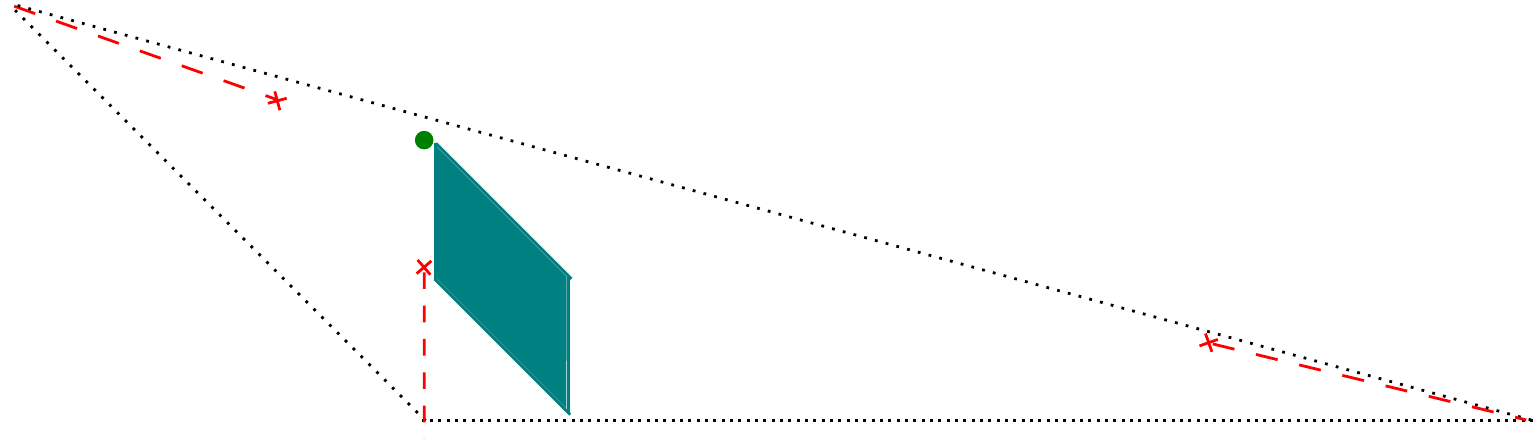}}

\caption{A ball
of capacity $\frac{\pi}{3}$ in the complement of the monotone fibre in an ATF of $\CP^2 \setminus E$ projects
into a arbitrarily small neighbourhood of the diamond $\Diamond(\sfrac{1}{6})$ (shaded).
}
\label{fig:Diamond_abc_ATF}

\end{center}
\end{figure}

Looking again back at Figure \ref{fig:abc_ATF}, we can indeed see
a diamond $\Diamond(\sfrac{1}{6})$ in the complement of $E$. A small neighbourhood
thereof contains a ball
of capacity $\frac{\pi}{3}$. (See Figure \ref{fig:Diamond_abc_ATF}.)
An application of Proposition \ref{prop:DiamondAlmostToric} gives rise to the following
initial estimate
\be\label{eq:prelim}
\inf_{(a,b,c) \in \frak M} c_G(\CP^2 \setminus E; T_{a,b,c}) \ge \frac{\pi}{3}.
\ee
We will give a better improved estimate later in Section \label{sec:CP2-E} after
a finer study of the base diagram associated to $(a,b,c)$.

\section{Geometry of the locus of the union of $T_{a,b,c}$}\label{sec:dense}

In this section, we construct a representative of the family $\{T_{a,b,c}\}$
of monotone Lagrangian tori such that the loci of the tori $T_{a,b,c}$ is not dense in $\CP^2$.

\subsection{Construction of a non-spreading family $\{T_{a,b,c}$\}}

We consider the configuration formed by the union of the Clifford torus and three Lagrangian disks,
as first exposed in \cite{STW16}, see also \cite{To17,PaTo17}.
We will construct a family such that all tori $T_{a,b,c}$ reside in an
arbitrarily small neighbourhood of the locus of this configuration.

We denote $\Sk$ for this configuration, which can also be thought as
a Lagrangian skeleton of the Liouville domain $M = \CP^2 \setminus E$.
This Lagrangian skeleton consists of the monotone Clifford torus $T_{\text{\rm Cl}}$ together
with three Lagrangian disks, with boundaries on $T_{\text{\rm Cl}}$.

To make our construction in perspective, we recall the notion of \emph{Lagrangian seeds}
from \cite{PaTo17}.

\begin{defn}[Definition 4.7 \cite{PaTo17}] A \emph{Lagrangian seed} $(L,\{D_i\})$ in a
symplectic 4-manifold $X$ consists of a monotone torus $L \subset X$, and a collection of
embedded Lagrangian disks $D_i \subset X$ with boundary on $L$, which satisfies the following
conditions. Here we denote $D_i^\circ = D_i \setminus \del D_i$.
\begin{enumerate}
\item each $D_i$ is attached to $L$ cleanly along its boundary, i.e., transversely in the
directions complementary to the tangent lines $T(\del D_i)$,
\item $D_i^\circ \cap L = \emptyset$,
\item $D_i^\circ \cap D_j^\circ = \emptyset, \, i \neq j$,
\item the curves $\del D_i \subset L$ have minimal pairwise intersections, i.e., there is a
diffeomorphism $L \to T^2$ taking each $\del D_i$ to a geodesic of the flat metric.
\end{enumerate}
\end{defn}

With this definition, the above mentioned configuration
$$
(T_{\text{\rm Cl}}, \{D_1, D_2, D_3\})
$$
as drawn in Figure \ref{fig:All_ATFs} is nothing but an example of Lagrangian seed.

The following is the precise statement on which we will be based for this inductive procedure starting from
$(1,1,1)$ to arbitrarily given $(a,b,c)$.

\begin{lem}[Compare with Lemma 4.16 \cite{PaTo17}] Denote $M = \CP^2 \setminus E$. Consider the Clifford torus $L_0$ and
a Lagrangian disk $D_0$ so that $(L_0,D_0) \subset M$ is a mutation configuration. Denote by
$\theta_M$ the Liouville one-form of the exact symplectic form $(\omega_{\text{\rm FS}})|_M$. Then
\begin{enumerate}
\item any neighborhood of $L_0 \cup D_0$ contains another mutation configuration $(L_1,D_1)$,
\item there is an arbitrarily small neighborhood $U_0$ of $L_0 \cup D_0$ such that
$(\theta_M)|_{U_0}$ is Liouville, and such that the completion of $U_0$ is isomorphic to
the completion of $M$.
\end{enumerate}
\end{lem}

 \begin{figure}[h!]

\begin{center}

\centerline{\includegraphics[scale=0.7]{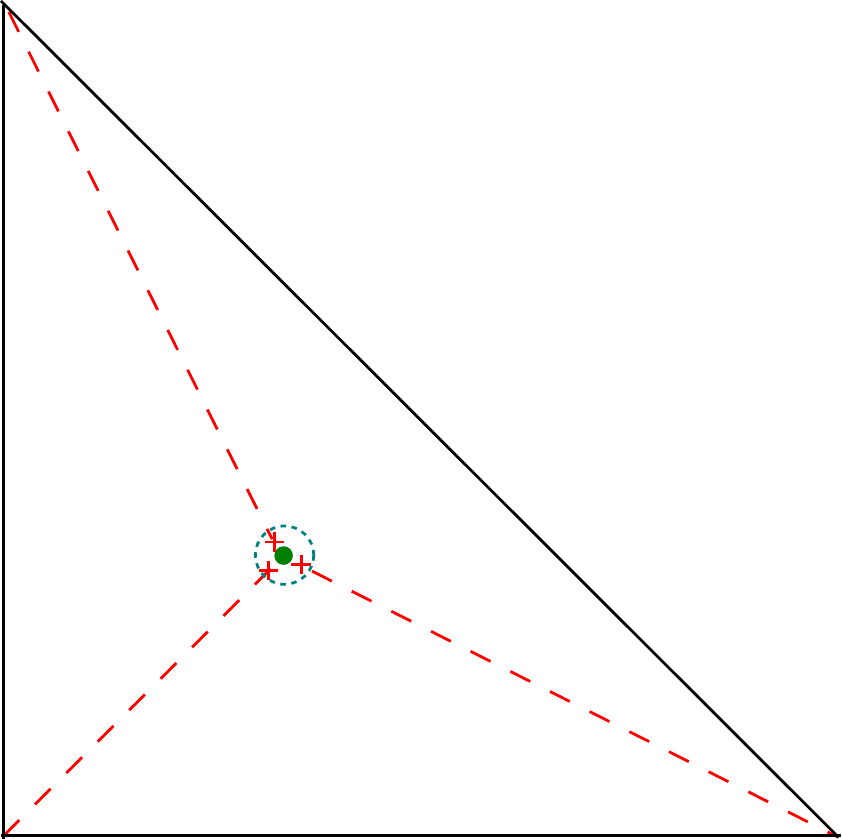}}

\caption{An ATF of $\CP^2$, with nodes very close to the monotone fibre. All mutations
can be made in $D$, the inside of the dashed circle, i. e., there are representatives of all
$T_{a,b,c}$ tori living inside $D$. The dashed circle can be infinitesimally small.
}
\label{fig:All_ATFs}

\end{center}
\end{figure}

The easiest way to see that we can make another mutation configuration
as close to the given on as we want is to slide all the nodes of an ATF very close to
the monotone fibre -- see Figure \ref{fig:All_ATFs}. Say that all the nodes
are now inside a small disk $D$ in the base. All mutations can then be
achieved by sliding the nodes inside $D$, so that the fibration remains
unchanged in the complement of $D$. In other words, all the monotone fibre
live in the pre-image of $D$ -- see Figure \ref{fig:All_ATFs}.

We would like to emphasize that this mutation operation is done in a way
that the ambient symplectic form, say, the Fubini-Study form unchanged.
In particular the union of all these tori is not dense in $\CP^2$
with respect to the Fubini-Study metric.
In fact, our construction shows that
the whole family $\{T_{a,b,c}\}$ can be put into an
open set of arbitrarily small volume by taking the above mentioned
neighborhood as small as we want.

\subsection{Ball packing in the complement of all these tori}

Inside the standard toric diagram of $\CP^2$, $\Sk$ projects into
the barycentre, union the three segments from the barycentre to the
vertices, illustrated as dashed lines in Figure \ref{fig: 9BallAll_ATFs}.

So, as Figure \ref{fig: 9BallAll_ATFs} also illustrates, we
can find 9 symplectic balls of the same size in the complement
of $\Sk$, hence in the complement of all $T_{a,b,c}$ Lagrangian tori,
for any capacity smaller than the capacity of the monotone ball.

 \begin{figure}[h!]
 \begin{center}

\centerline{\includegraphics[scale=0.7]{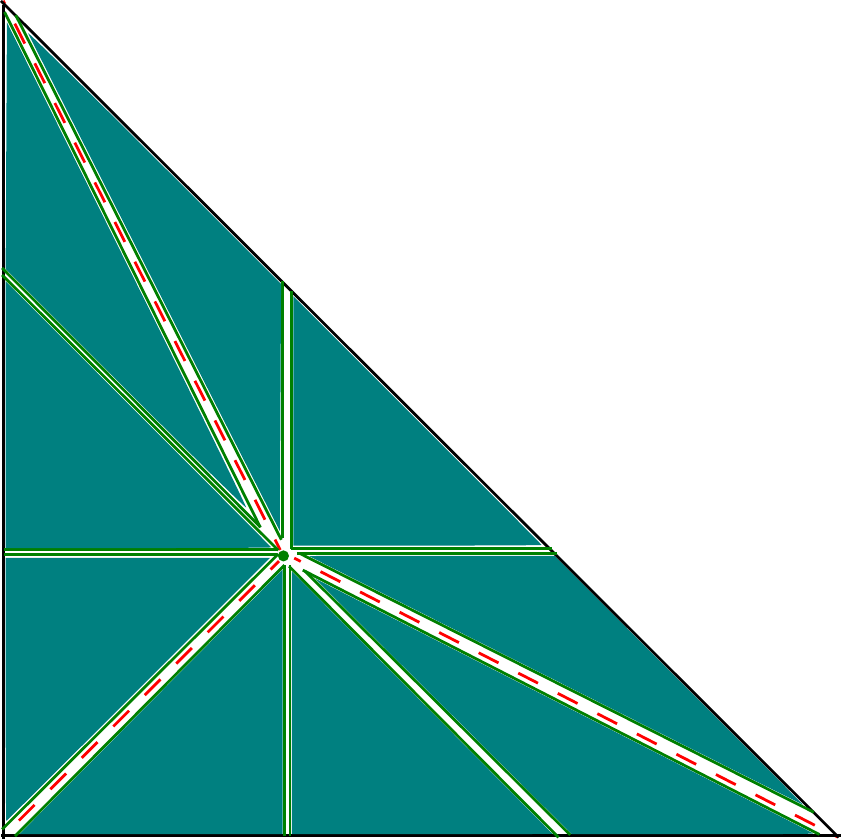}}

\caption{For any given capacity $c$ smaller than the capacity of the monotone ball,
we can find 9 balls in the complement of $\Sk$, and hence
in the complement of representatives of all
$T_{a,b,c}$ tori.}

\label{fig: 9BallAll_ATFs}

\end{center}
\end{figure}

This is the maximum we can get for volume reasons.

\section{Ball packing in the complement of individual torus $T_{a,b,c}$ }

We can easily see three monotone balls of the same size in the complement of the Clifford torus
in $\CP^2$, via the toric blowup.

In \cite[Section~7]{ChSch10}, Chekanov-Schlenk ask if one can embed the Chekanov
torus into the monotone $\BlIII$. The answer is yes and we can indeed
easily find two extra monotone balls in the complement
of the Chekanov torus inside $\BlII$ -- see Figure \ref{fig: Chek_Ball_CP2}.

 \begin{figure}[h!]
 \begin{center}

\centerline{\includegraphics[scale=0.7]{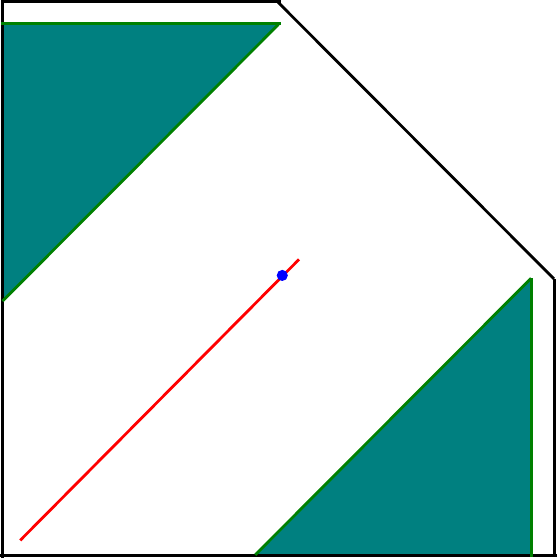}}

\caption{Two extra monotone balls in the complement of the Chekanov torus in $\BlII$.
The Chekanov torus projects over the red segment -- see \cite[Section~7]{ChSch10}.}

\label{fig: Chek_Ball_CP2}

\end{center}
\end{figure}

\begin{figure}[h!]
 \begin{center}

\centerline{\includegraphics[scale=0.7]{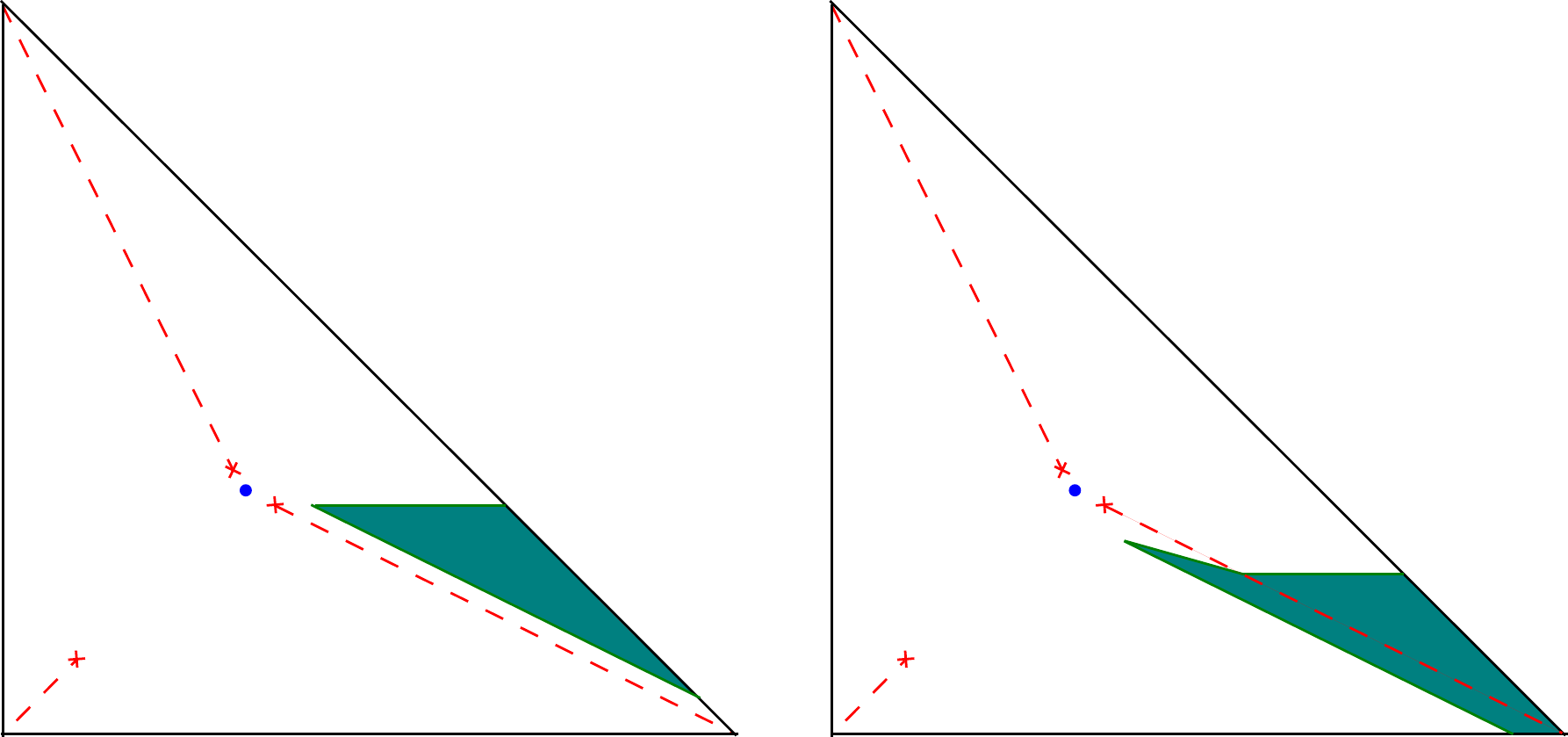}}

\caption{Two triangles of the same size in a ATF diagram of $\CP^2$. The second can be thought
to be obtained from the first by ``sliding it through the cut''. Note that one edge corresponds to
the eigen-direction of the cut, hence it is not distorted as it passes through.}

\label{fig:ExampBall_CP2}

\end{center}
\end{figure}

We recall that in an almost toric fibration, we do have the fibres over the
cuts, only the affine structure of the base diagram is not corresponding
with the standard affine structure of $\R^2$. In particular, we can
have a triangular region passing through the cut -- see Figure
\ref{fig:ExampBall_CP2}. The monodromy may distort the edges of the
triangle as it crosses the cut.

This allow us to get even better results for the ball packing.
Start with a configuration of 5 consecutive balls, similar to the ones in Figure \ref{fig:
9BallAll_ATFs}. Consider an ATF with two cuts very close to the monotone Cliford
torus. Slide this 5 balls through the cuts as illustrated in Figure
\ref{fig:ExampBall_CP2}. You can then ``inflate the triangles'', so
all of them become monotone triangles. You get a diagram as illustrated in
Figure \ref{fig: 5Cliff_Ball_CP2}.

 \begin{figure}[h!]
 \begin{center}

\centerline{\includegraphics[scale=0.7]{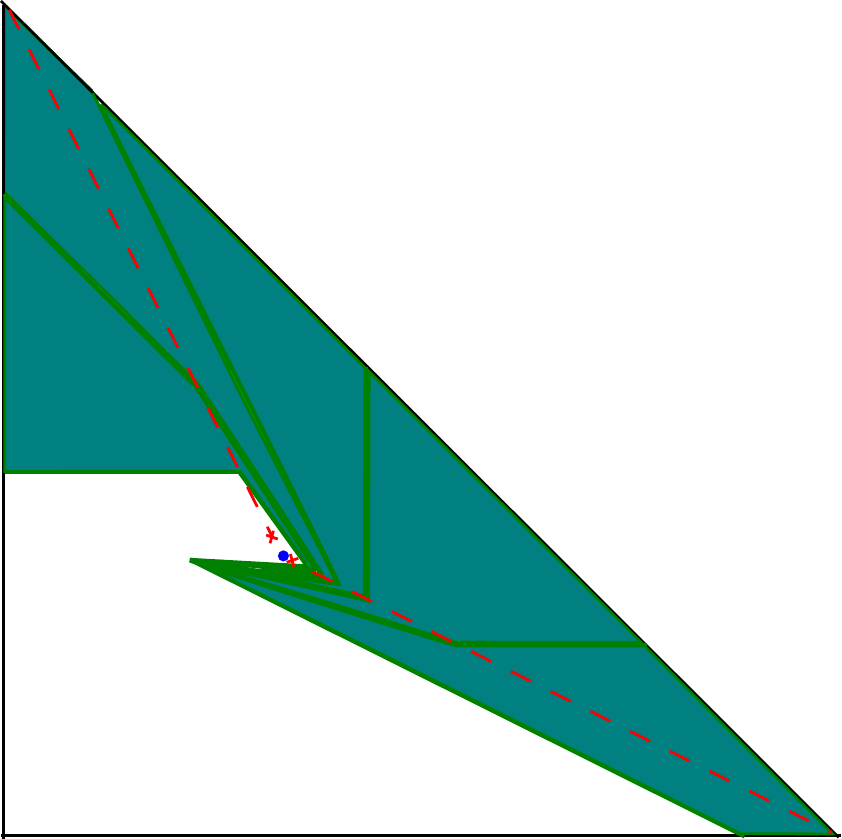}}

\caption{Five monotone balls in the complement of representatives of $T_{1,b,c}$
tori.}

\label{fig: 5Cliff_Ball_CP2}

\end{center}
\end{figure}

  \begin{figure}[h!]
 \begin{center}

\centerline{\includegraphics[scale=0.65]{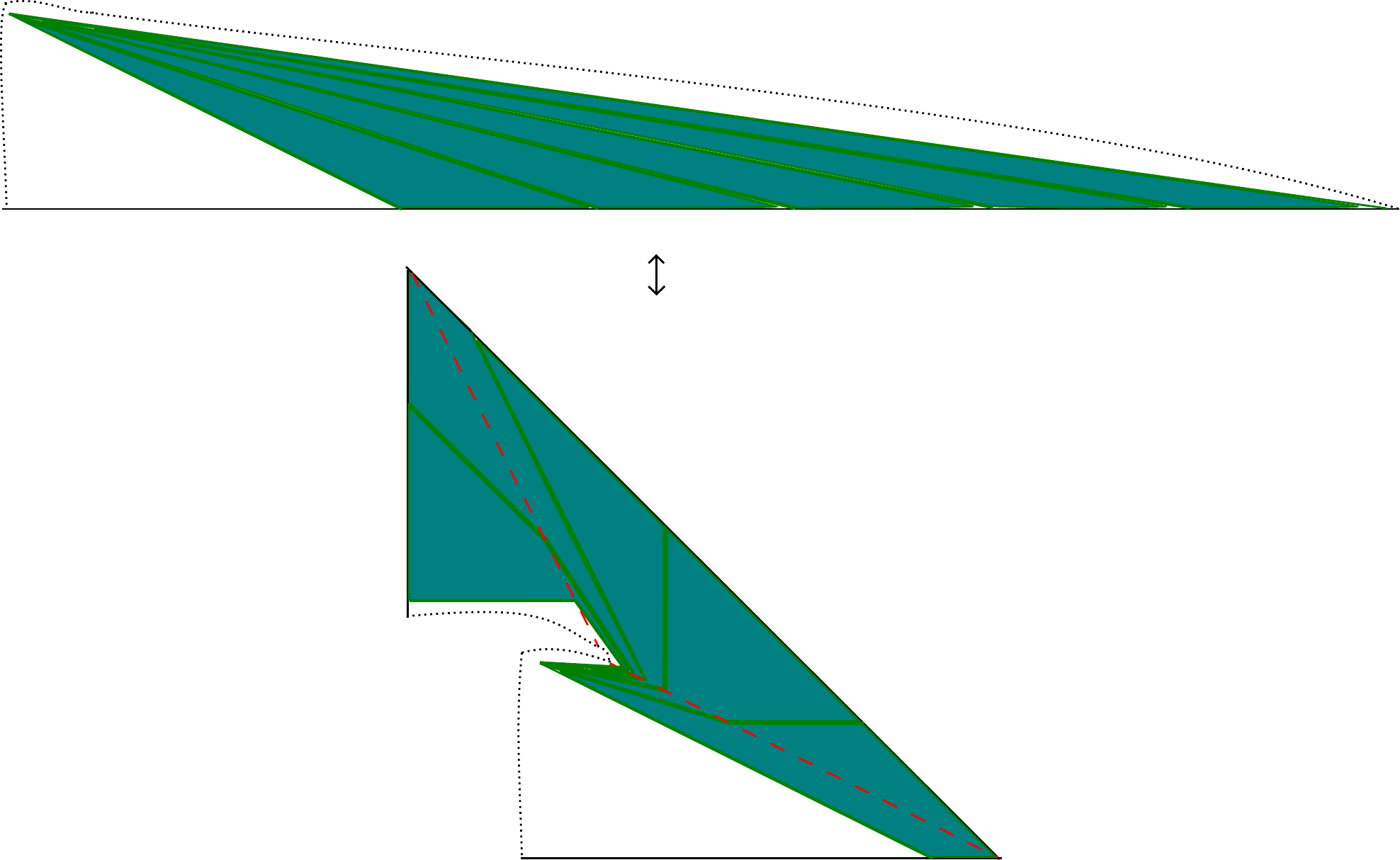}}

\caption{On the top, a configuration of five monotone balls. On the bottom, the
same configuration drawn after introducing two cuts.}

\label{fig: 5Ball_Config}

\end{center}
\end{figure}

{Of course, when we
want to embed the monotone balls, we need a tiny neighbourhood of the monotone
triangle. So all these triangles need to be spaced out by a tiny amount, which
is not drawn on Figure \ref{fig: 5Cliff_Ball_CP2} for visual purposes.
Figure \ref{fig: 5Ball_Config} illustrates how the balls look like
when we undo the monodromies associated with the cuts for better
understanding.

 These 5 monotone balls are indeed in the
complement of tori of the form $T_{1,b,c}$ $(1 + b^2 + c^2 = 3bc)$, all
together, in particular of both Clifford and Chekanov tori. This is because
all these tori are obtained by changing the ATF in the pre-image of a small
region containing the monotone fiber and the two singular fibres --
recall the analogous discussion given in Figure \ref{fig:All_ATFs}.

\begin{rmk}
We cannot use this trick in the ATF illustrated in Figure \ref{fig:All_ATFs},
 to get a monotone ball in the complement of $\Sk$ and, hence, of all tori $T_{a,b,c}$
 simultaneously.
  If you try to ``slide one triangle of Figure \ref{fig: 9BallAll_ATFs}
  through a cut'', with a triangle of size close to the monotone one,
  it will be forced to cross all the three cuts several times, in a spiral fashion,
  before eventually entering the dashed neighbourhood in Figure \ref{fig:All_ATFs}.
\end{rmk}

\vspace{2cm}

To proceed further, we need to make some computations regarding Markov triples $(a,b,c)$,

\[a^2 + b^2 + c^2 = 3abc. \]

The following is well-known

\begin{lem} \label{prp:a'}
  If $c \ge a$ the $a' = 3bc - a > c$.
\end{lem}

\begin{proof}
  $c \ge a \imp bc \ge a \iff a' = 3bc - a \ge 2bc  \imp a' > c$ since $2b \ge 2$.
\end{proof}

 In particular, if $a\le b \le c$, then $a' = 3bc -a > c$ and $b' = 3ac - b > c$.

 \begin{lem}[Section 3.7 of \cite{KaNo98}] \label{prp:KaNo2}
 Two out of the
  three possible mutations of the Markov triple $(a,b,c)$ increase the sum $a + b + c$ and the other reduces it.
\end{lem}

(In fact,  we can indeed deduce from this that if $a \le b < c$ then  $c' \le b$, but we won't need
that.)

 \begin{prop} \label{prp:2/3}
Suppose  $c > b \ge a$. Then
 \be\label{eq:ge2/3}
 \frac{c^2}{a^2 + b^2 + c^2} \ge \frac{2}{3}.
 \ee
 \end{prop}
\begin{proof} We first transform \eqref{eq:ge2/3} as follows:

    \[ \frac{c^2}{a^2 + b^2 + c^2} \ge \frac{2}{3}  \iff c^2 \ge 2(a^2 + b^2)
    \iff c^2 \ge 2(3ab - c)c  \iff  \]

\[ \iff 6ab - 2c \le c \iff 2ab \le c. \]

We will prove, by induction on the biggest
Markov number, that if $c \ge 2$ then  $2ab \le c$. This holds for our base $(a,b,c)=(1,1,2)$.

Suppose that $2ab \le c$ for $a \le b < c$ and consider mutations that increase the
biggest Markov number in the triple.
We derive from Propositions \ref{prp:a'} and \ref{prp:KaNo2} that the mutations that
increase the biggest Markov number in the triple are $a \leftrightarrow a' = 3bc
-a$ and $b \leftrightarrow b' = 3ac - b$ being $a'$ the biggest number in the
triple $(b,c,a')$ and $b'$ the biggest number in the triple $(a,c,b')$.

So we need to show that
\[ 2bc \le a' \ \ \ \  2ac \le b'  \]
But we already saw that in the proof of Proposition \ref{prp:a'}:

 \[ c \ge a \imp bc \ge a \iff a' = 3bc - a \ge 2bc \]

  \[ c \ge b \imp ac \ge b \iff b' = 3ac - b \ge 2ac \]
This finishes the proof.
 \end{proof}

 The inequality \eqref{eq:ge2/3} means that, if the affine lengths of the edges are
 $a^2$, $b^2$, $c^2$, then the longest edge $c^2$ has at least
 $\sfrac{2}{3}$ of the sum of the affine lengths $a^2 + b^2 + c^2$.

The sum of lengths of the edges is $9$ times the size of the base of the
monotone triangle. This means that the longest edge has at least $9\frac{2}{3} = 6$
times the length of the base of the monotone triangle. Hence we can
see at least 5 monotone balls in the complement of $T_{a,b,c}$, for $c\ge 2$,
see Figure \ref{fig:5Ball_abc_ATF}.

 Since we have already showed in Figure \ref{fig: 5Cliff_Ball_CP2} that we can find
 5 monotone balls in the complement of the Clifford torus
 $T(1,1,1)$, we derive that we can find
 5 monotone balls in the complement of the union $T_{a,b,c}$ for all $(a,b,c) \in \frak M$.

 \begin{figure}[h!]

\begin{center}

\centerline{\includegraphics[scale=0.65]{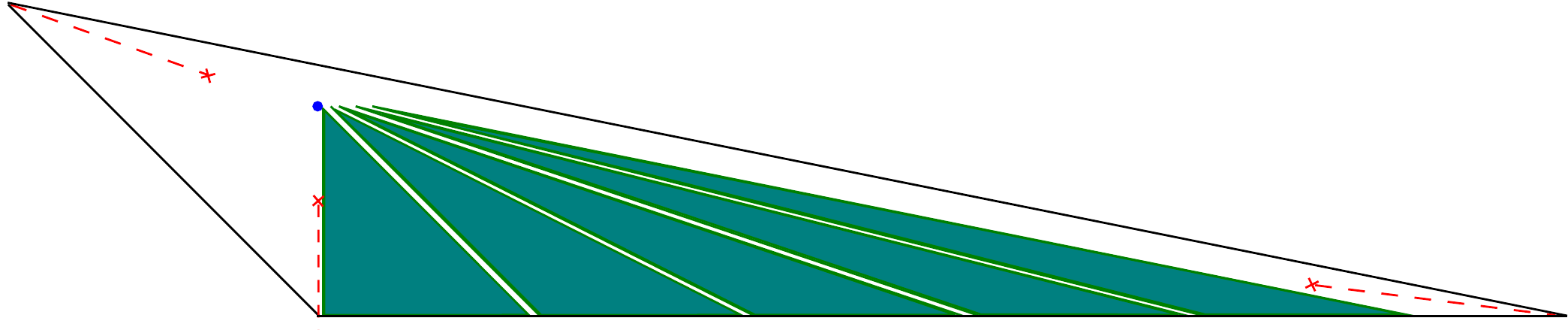}}

\caption{5 monotone balls in the complement of the monotone $T_{a,b,c}$ fibre in an ATF of $\CP^2$,
for $c\ge 2$.}
\label{fig:5Ball_abc_ATF}

\end{center}
\end{figure}

We summarize the above discussion into the following
\begin{thm}\label{thm:5balls} Any $T_{a,b,c}$ tori, in particular the Chekanov torus, can be embedded into
the monotone $\CP^2 \# \overline{k\CP^2}$ for $k \leq 5$.
\end{thm}

This affirmatively answers to a question asked by Chekanov and Schlenk \cite[Section 7]{ChSch10}.
(See Theorem \ref{thm:Chekanov-Schlenk} and the discussion around it.)

\section{In the complement of an elliptic curve}
\label{sec:CP2-E}

By \cite[Proposition~8.2]{Sy03}, we know that the preimage E of the
edges of a almost toric fibration diagram, with no rank $0$ singularities
(i.e., all nodes pushed inside) is a smooth symplectic torus representing
the anticanonical divisor. By a result of Sikorav \cite[Theorem~3]{Si03},
see also \cite{ST05}, we can assume that this boundary is indeed
an elliptic curve.

In this section we improve the estimate \eqref{eq:prelim} further combining the
results from the previous section.

 From Proposition \ref{prp:2/3}, we see that if $c \ge 2$ we can embed,
 in the complement of $T_{a,b,c}$ in an ATF of $\CP^2 \setminus E$,
 triangles as close as we need to the triangle of height equal to the height of
 the monotone triangle and base equal 6 times the base of the monotone triangle
 -- see Figure \ref{fig:Diam6_CP2abc}.

   \begin{figure}[h!]

\begin{center}

\centerline{\includegraphics[scale=0.5]{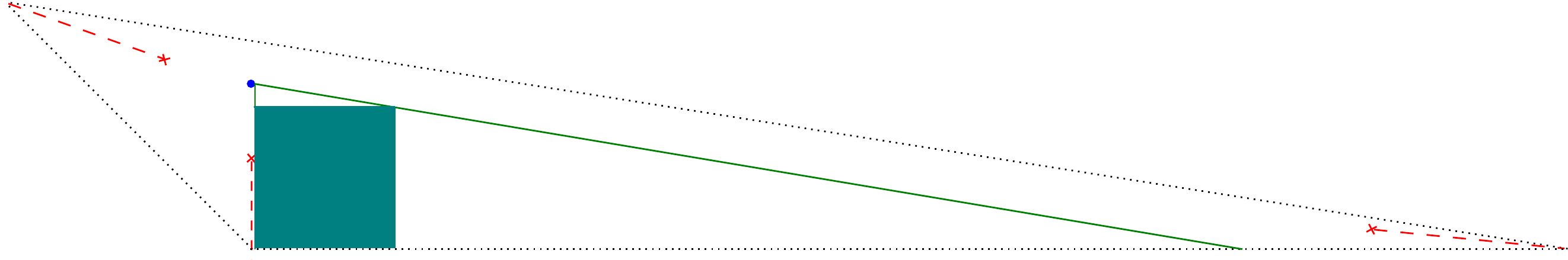}}

\caption{For $c\ge 2$ we can get balls of capacities converging to $\frac{4\pi}{7}$ in the complement
of any $T_{a,b,c}$ monotone fibre in an ATF of $\CP^2 \setminus E$ projecting
into diamonds converging to $\Diamond(\sfrac{2}{7})$ (shaded).
}
\label{fig:Diam6_CP2abc}

\end{center}
\end{figure}

 Inscribed inside one of these triangles we can embed a square of sides with
 length as close to $6/7$ the height of the monotone triangle as we want, i.e.,
 a diamond $\Diamond(\frac{6}{7}\frac{1}{3}) = \Diamond(\frac{2}{7})$ -- see
 Figure \ref{fig:Diam6_CP2abc}. Hence we can embed a ball of capacity as close
 to $\frac{4\pi}{7}$ as we want.

 \begin{figure}[h!]

\begin{center}

\centerline{\includegraphics[scale=0.7]{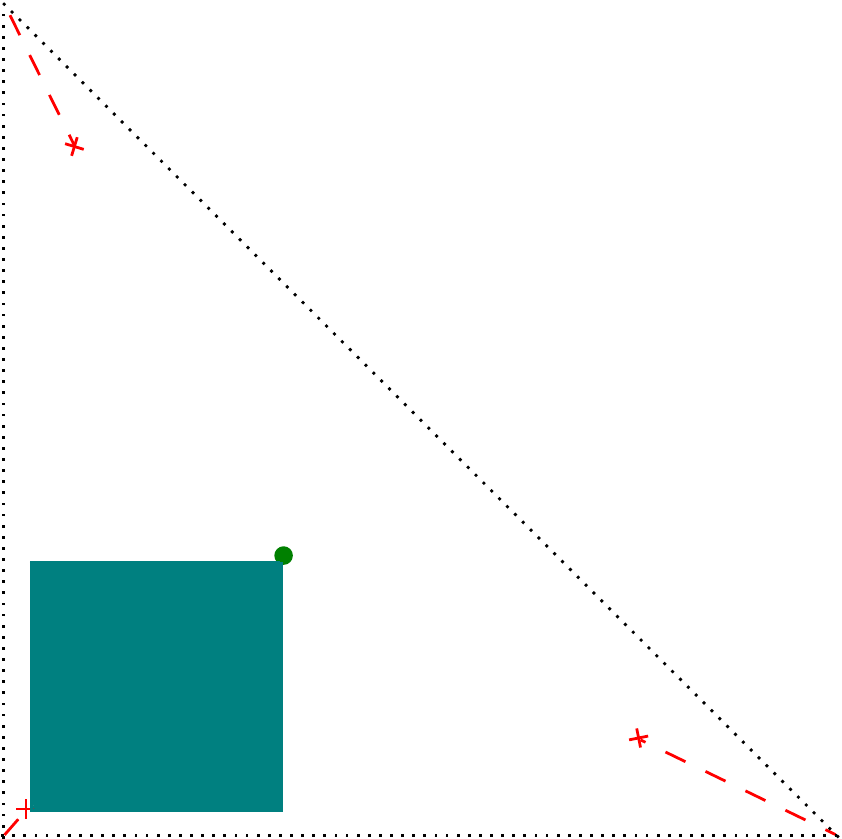}}

\caption{Balls of capacities converging to $\frac{2\pi}{3}$ in the complement
of the monotone Clifford torus fibre in an ATF of $\CP^2 \setminus E$.
}
\label{fig:DiamCliff_ATF_CP2}

\end{center}
\end{figure}

Moreover, in the complement of the Clifford torus $T_{\text{\rm Cl}} \cong T_{1,1,1}$ in $\CP^2 \setminus
E$ we can get diamonds as close to $\Diamond(\frac{1}{3})$ as we want,
hence with capacities converging to $\frac{2\pi}{3}$ -- see Figure \ref{fig:DiamCliff_ATF_CP2}.

We summarize the above discussion into the following

\begin{thm}\label{thm:final}

 \[ \inf_{(a,b,c) \in \frak M} c_G(\CP^2 \setminus E; T_{a,b,c}) \ge \frac{4\pi}{7} > \frac{\pi}{3}  \]

\end{thm}

\section{Discussion and open questions}\label{sec:discussion}

In this section, we would like to propose a few interesting open questions
in relation to the geometry of $T_{a,b,c}$ in addition to Problem \ref{prb:Tabc}.

\subsection{Hamiltonian-minimal representative}

As mentioned in the introduction, the starting point of our research in the present article
lies in our attempt to visualize the $T_{a,b,c}$ tori in terms of the geometry of Fubini-Study metric
on $\CP^2$. One interesting question is to find a geometric realization of the tori $T_{a,b,c}$ with
special Riemannian geometric properties in the spirit of \cite{Oh90,Oh93B,Oh94}.

\begin{ques} \begin{enumerate}
\item Construct a \emph{Hamiltonian-minimal} representative in each Hamiltonian
isotopy class of $T_{a,b,c}$ with respect to the Fubini-Study metric and visualize the
representative.
\item More boldly prove that there exists a Hamiltonian-minimal representative of $T_{a,b,c}$
in its Hamiltonian isotopy class, or that there is a lower bound of the volume inside
the Hamiltonian isotopy class.
\item Is there any alternative group theoretic construction of $T_{a,b,c}$?
\end{enumerate}
\end{ques}

According to \cite{Oh94}, the mean curvature flow of Lagrangian submanifolds in
Einstein-K\"ahler manifolds such as in $\CP^2$ equipped with Fubini-Study metric $g$ preserves
the Lagrangian property and decreases the volume.

On the other hand, it follows from \cite{Oh93B} and
the index calculation given by Urbano \cite{Ur93} that
any \emph{Hamiltonian-stable} Hamiltonian-minimal Lagrangian torus with respect to the
Fubini-Study metric is isometric to the Clifford torus in $\CP^2$, \emph{provided it is
smooth}. Therefore none of smooth representative $T_{a,b,c}$ are Hamiltonian-stable unless
$(a,b,c) = (1,1,1)$.
This in particular implies that there is no \emph{smooth} volume minimizing representative of $T_{a,b,c}$
in its Hamiltonian isotopy class unless $(a,b,c) = (1,1,1)$. It was proven in \cite{Oh93B} that
the Clifford torus is volume minimizing under a sufficiently small Hamiltonian isotopy, and in
particular it is Hamiltonian-stable. These observations reveal that
the second question may be a quite hard but interesting question to ask.
Existence of a positive lower bound
of the volume inside each given Hamiltonian isotopy class can be proved via the
Crofton's formula (see \cite[Introduction]{Oh90}),
if the following question is affirmative

\begin{ques} Consider the set $\{g \cdot \R P^2\mid g \in \text{\rm Iso}(\CP^2)\}$, i.e.,
the set of totally geodesic $\R P^2$. Is is true that $T_{a,b,c} \cap (g \cdot \R P^2) \neq \emptyset$
for all $g \in \text{\rm Iso}(\CP^2)$. Here $\text{\rm Iso}(\CP^2)$ denotes the isometry group
of the Fubini-Study metric.
\end{ques}

Since $\text{\rm Iso}(\CP^2) \subset \text{\rm Ham}(\CP^2)$, the above mentioned non-intersection
result follows from the Floer theoretic question whether or not $T_{a,b,c} \cap \phi(\R P^2) \neq \emptyset$
for all $\phi \in \text{\rm Ham}(\CP^2)$. It turns out that this intersection result
depends on the types of Markov triples $(a,b,c)$. For example, Alston-Amorim \cite{AA12}
proved that the Clifford torus $T_{\text{\rm Cl}} \cong T_{1,1,1}$ intersects $\phi(\R P^2)$ for all $\phi$:
They proved that a version of Floer cohomology between the product $T_{\text{\rm Cl}}
\times T_{\text{\rm Cl}}$ and $\R P^2 \times \R P^2$ in $\CP^2 \times \CP^2$ is defined
and is non-zero \emph{even though the Floer cohomology between $\R P^2$ and $T_{\text{\rm Cl}}$
is not defined.}

On the other hand, the case $(a^2,b^2,c^2) = (1,1,4)$
i.e., $T_{1,1,2}$ was studied by Wei-Wei Wu \cite{Wu15} for which the fiber at the singular
vertex of the moment polytope is symplectomorphic to $\R P^2$ and so it does not intersect
the semi-toric fiber $T(1,1,4) \cong T_{1,1,2}$. (It is also shown in \cite{OU13} that
$T_{1,1,2}$ is the Chekanov torus.) In fact, such a non-intersection result
can be proved for any triple $(a,b,c)$ one of whose element is $2$. This can be seen
by mutating the smooth vertices in Wu's semi-toric picture \cite{Wu15}, for instance.

These observations lead us to proposing the
following conjecture which is an interesting subject of future investigation.

\begin{conj}\label{conj:RP2Tabc} There exists a Hamiltonian diffeomorphism $\phi$ on
$\CP^2$ such that $T_{a,b,c} \cap \phi(\R P^2) = \emptyset$ if and only if
$(a,b,c) = (2,b,c)$.
\end{conj}

Finally, under the assumption that the answer to the second question above
is affirmative, the following question is interesting to ask.

\begin{ques}\label{ques:asymp-volume} Is there a family $\{T_{a,b,c}\}$ such that
there exists a positive constant $\epsilon_\infty > 0$
$$
\inf_{a,b,c \in \frak M} \frac{\vol_g(T_{a,b,c};\CP^2)}{(abc)^{2/3}} \geq \epsilon_\infty?
$$
\end{ques}

\subsection{Size of Weinstein neighborhood of $T_{a,b,c}$}

Another question is related to the size of the maximal Darboux-Weinstein
neighborhood of $T_{a,b,c}$. We start with some general discussion on Darboux-Weinstein chart.
Let $L \subset M$ be a compact
Lagrangian submanifold. Consider the Darboux-Weinstein chart
$
\Phi: \CU \to \CV
$
where $\CU$ is a neighborhood of $L$ in $M$ and $\CV$ is a
neighborhood of the zero section $o_L \subset T^*L$. Then by definition,
we have
$$
\omega = \Phi^*\omega_0, \quad \omega_0 = - d\theta
$$
for the Liouville one-form $\theta$ on $T^*L$ and $\Phi|_L = id_L$
under the identification of $L$ with $o_L$.

Fix any Riemannian metric $g$ on $L$. For $x \in \CU$, we define
$$
\|x\|_{g,\Phi} = \|\Phi(x)\|_{g(\pi(\Phi(x))}
$$
where $\Phi(x) \in T_{\pi(\Phi(x))}^* L$ and $\pi: T^*L \to L$
is the canonical projection, and $\|\cdot \|_{g(q)}$ is the
norm on $T_q^*L$ induced by the inner product $g(q)$.

\begin{defn}\label{defn:width} Let $L \subset M$ be a compact
Lagrangian submanifold equipped with a metric $g$. Consider the Darboux-Weinstein chart
$
\Phi: \CU \to \CV.
$
Define
$$
{\frak w}_{\text{DW}}(\Phi;g): = \inf_{q \in L} \left(\sup_{x \in \pi^{-1}(q) \cap \CU} \|x\|_{g,\Phi}\right)
$$
and
\be\label{eq:DW-width}
{\frak w}_{\text{DW}}(L;M) = \sup_{\Phi} {\frak w}_{\text{DW}}(\Phi;g)
\ee
over all Darboux-Weinstein chart of $L$.
We call ${\frak w}_{\text{DW}}(L;M)$ the \emph{Weinstein width} of $L$
(relative to the metric $g$).
\end{defn}
The ${\frak w}_{\text{DW}}(L;M)$ is another symplectic invariant of $L$ which measures
extrinsic complexity of the embedding $L \subset M$.
Obviously ${\frak w}_{\text{DW}}(L;M) > 0$
since ${\frak w}_{\text{DW}}(\Phi;g) > 0$ for any Darboux-Weinstein chart $\Phi$
for compact Lagrangian submanifold $L$.

With this preparation, we propose the following conjecture which is another way of
examining the conjectural \emph{ergodic} behavior of the family $T_{a,b,c}$ .

\begin{conj}\label{conj:Weinstein-chart}
Consider $T_{a,b,c} \subset \CP^2$. Then
$$
\inf_{(a,b,c) \in \frak M}
{\frak w}_{\text{DW}}(T_{a,b,c};\CP^2) = 0.
$$
\end{conj}

Proving this conjecture is essentially equivalent to
proving the infimum over $(a,b,c) \in \frak M$ of the size of the shape invariant
$Sh_{T_{a,b,c}}(\CP^2)$ is zero. See \cite{Si89} for the definition of
the shape invariant and \cite[Section 6]{ShToVi18} for the relevant study of this
shape invariants.

\bibliographystyle{alphanum}
\bibliography{SympRefs}

\end{document}